\documentclass[11pt,a4paper]{article}
\usepackage{amsmath, amssymb, amsthm}
\usepackage{enumitem}
\usepackage{geometry}
\usepackage{boondox-cal}
\usepackage{graphicx}
\usepackage{url}
\usepackage{tikz-cd}

\newtheorem{theorem}{Theorem}[section]
\newtheorem{proposition}[theorem]{Proposition}
\newtheorem{lemma}[theorem]{Lemma}
\newtheorem{corollary}[theorem]{Corollary}
\theoremstyle{definition}
\newtheorem{definition}[theorem]{Definition}
\newtheorem{example}[theorem]{Example}
\newtheorem{remark}[theorem]{Remark}

\newcommand\blfootnote[1]{%
  \begingroup
  \renewcommand\thefootnote{}\footnote{#1}%
  \addtocounter{footnote}{-1}%
  \endgroup
}

\title{Gray $\mathcal{A}$-metrics on nilpotent Lie groups and the quasicentroid condition}
{\author{
  \begin{tabular}{@{}c@{}}
\textsc{Fern\'andez-Culma}, Edison Alberto \textsuperscript{$\dagger$}\blfootnote{${}^{\dagger}$ Partially supported by CONICET and SECyT-UNC: 33620180100523CB} \\
    {\small Centro de Investigaci\'on y Estudios en Matem\'atica de C\'ordoba, CONICET}\\
    {\small Facultad de Matem\'atica, Astronom\'ia, F\'isica y Computaci\'on, UNC}\\
    {\small C\'ordoba, Argentina}\\
    \texttt{\small efernandez@famaf.unc.edu.ar}\\
  \end{tabular}
}}

\begin{document}

\maketitle


%%%%%%%%%%%%%%%%%%%%%%%%%%%%%%%%%%%%%%%%%%%%%%%%%%%%%%%%%%%%%%%%%%%
% ABSTRACT
%%%%%%%%%%%%%%%%%%%%%%%%%%%%%%%%%%%%%%%%%%%%%%%%%%%%%%%%%%%%%%%%%%%

\begin{abstract}
We study left-invariant Gray $\mathcal{A}$-metrics (cyclic-parallel Ricci tensor) on nilpotent Lie groups, focusing on the algebraic condition that the Ricci operator belongs to the quasicentroid $\operatorname{QC}(\mathfrak{n})$ of the Lie algebra $\mathfrak{n}$. We review and develop the abstract structure theory of quasicentroids, and prove a Central Decomposition Theorem: any nilpotent Riemannian Lie algebra with $\operatorname{Ric} \in \operatorname{QC}(\mathfrak{n})$ decomposes as a central product of ideals $\mathfrak{i}_j = E_{\lambda_j} \oplus \mathfrak{z}$, where $\mathfrak{z}$ is the center of $\mathfrak{n}$ and $E_{\lambda_j}$ are the negative Ricci eigenspaces. Endowed with the restricted metric, each factor $\mathfrak{i}_j$ is itself a Gray $\mathcal{A}$-space. Finally, we investigate when the general $\mathcal{A}$-metric condition reduces to $\operatorname{Ric} \in \operatorname{QC}(\mathfrak{n})$. We prove this equivalence unconditionally for two-step nilpotent algebras and whenever the metric admits a nice basis, and provide a general characterization via the vanishing of two trace identities.
\end{abstract}

%%%%%%%%%%%%%%%%%%%%%%%%%%%%%%%%%%%%%%%%%%%%%%%%%%%%%%%%%%%%%%%%%%%
% INTRODUCTION
%%%%%%%%%%%%%%%%%%%%%%%%%%%%%%%%%%%%%%%%%%%%%%%%%%%%%%%%%%%%%%%%%%%

\section{Introduction}

In his seminal work on the classification of Einstein-like Riemannian manifolds, Alfred Gray \cite{Gray1978} introduced two fundamental classes of metrics, denoted by $\mathcal{A}$ and $\mathcal{B}$, which naturally interpolate between the class $\mathcal{E}$ of Einstein manifolds, the class $\mathcal{P}$ of manifolds with parallel Ricci tensor ($\nabla \operatorname{ric} = 0$), and the class $\mathcal{C}$ of manifolds with constant scalar curvature. The class $\mathcal{B}$ is characterized by metrics having a \textit{Codazzi Ricci tensor} (equivalently, those with harmonic curvature). On the other hand, a Riemannian manifold $(M,g)$ is a Gray $\mathcal{A}$-manifold if its Ricci tensor is cyclic-parallel; that is,
\begin{equation*}
    (\nabla_X \operatorname{ric})(Y,Z) + (\nabla_Y \operatorname{ric})(Z,X) + (\nabla_Z \operatorname{ric})(X,Y) = 0,
\end{equation*}
or equivalently, $(\nabla_X \operatorname{ric})(X,X) = 0$ for all tangent vector fields $X, Y, Z \in \mathfrak{X}(M)$.

From the standpoint of representation theory, the space of tensors sharing the algebraic symmetries of $\nabla \operatorname{ric}$ decomposes under the action of the orthogonal group into three irreducible components. Two of these components correspond exactly to the geometries defined by classes $\mathcal{A}$ and $\mathcal{B}$. This algebraic decomposition yields the strict inclusion hierarchy \cite{Gray1978}:
\begin{equation*}
\begin{tikzcd}[column sep=1.5em, row sep=0.5em]
& & \mathcal{A} \arrow[dr, phantom, sloped, "\subset"] & & & \\
\mathcal{E} \arrow[r, phantom, "\subset"] & \mathcal{P} = \mathcal{A} \cap \mathcal{B} \arrow[ur, phantom, sloped, "\subset"] \arrow[dr, phantom, sloped, "\subset"] & & \mathcal{A} \cup \mathcal{B} \arrow[r, phantom, "\subset"] & \mathcal{C} \arrow[r, phantom, "\subset"] & \mathcal{R} \\
& & \mathcal{B} \arrow[ur, phantom, sloped, "\subset"] & & &
\end{tikzcd}
\end{equation*}
where $\mathcal{R}$ denotes the class of all Riemannian manifolds.

Beyond its representation-theoretic origin, the class $\mathcal{A}$ plays a central role in the geometry of geodesic reflections. Motivated by Cartan's symmetric spaces, Joseph D'Atri and Helen Nickerson \cite{DatriNickerson1969, DatriNickerson1974} initiated the study of Riemannian manifolds whose local geodesic symmetries preserve the volume element up to sign, now known as \emph{D'Atri spaces} (see the survey \cite{KowalskiPruferVanhecke1996}). Being a D'Atri space is equivalent to satisfying an infinite sequence of curvature identities known as \textit{the odd Ledger conditions} $L_{2k+1}$ ($k \ge 1$). Crucially, the first non-trivial odd Ledger condition, $L_3 = 0$, coincides precisely with Gray's $\mathcal{A}$-condition $(\nabla_X \operatorname{ric})(X,X) = 0$. Consequently, every D'Atri space (in particular, every naturally reductive homogeneous space \cite{DatriNickerson1974}) is necessarily a Gray $\mathcal{A}$-manifold.

Because imposing the full infinite sequence of Ledger conditions directly is algebraically intractable, the standard strategy for classifying D'Atri spaces in low dimensions consists in first classifying all metrics in Gray's class $\mathcal{A}$ ($L_3 = 0$) and subsequently filtering the resulting families through higher-order Ledger obstructions. As detailed by Teresa Arias-Marco \cite{AriasMarco2007}, this classification program is notoriously delicate. For instance, initial attempts to classify four-dimensional homogeneous $\mathcal{A}$-manifolds and D'Atri spaces proved to be incomplete, omitting certain non-symmetric irreducible families, and ultimately required a definitive correction \cite{AriasMarco2007}.

By a classical theorem of John Milnor \cite[Theorem 2.4]{Milnor1976}, a non-abelian nilpotent Lie group never admits left-invariant Einstein metrics. Consequently, such groups cannot admit metrics with a parallel Ricci tensor ($\nabla \operatorname{ric} = 0$), as the parallel eigenspaces would force the Lie algebra to decompose as an orthogonal direct sum of Einstein ideals. Also, it is proved in \cite{DAtri1985} that any left-invariant metric with harmonic curvature on a nilpotent Lie group is necessarily Ricci-parallel. Thus, the Gray class $\mathcal{B}$ is empty in the non-abelian nilpotent setting. In stark contrast, nilpotent Lie groups support a rich variety of genuine left-invariant $\mathcal{A}$-metrics:

Aroldo Kaplan \cite{Kaplan1983} proved that all groups of Heisenberg type are D'Atri spaces, and hence $\mathcal{A}$-manifolds, even when they fail to be naturally reductive, while Carolyn Gordon \cite{Gordon1985} showed that any naturally reductive nilmanifold is at most two-step nilpotent. Very recently, systematic investigations of left-invariant $\mathcal{A}$-metrics on nilpotent Lie groups have been undertaken by Teresa Arias-Marco and Jos\'e Manuel Fern\'andez-Barroso \cite{AriasMarcoFB2026} and by Huibin Chen, Zhiqi Chen, Hui Zhang and Fuhai Zhu \cite{ChenChenZhu2026, ChenChenZhangZhu2025}, who studied the two-step case, under the name of \textit{generalized Heisenberg-type algebras}, and classified left-invariant $\mathcal{A}$-metrics on indecomposable nilpotent Lie groups up to dimension five.

The primary goal of the present article is to study a distinguished algebraic family of Gray $\mathcal{A}$-metrics on nilpotent Riemannian Lie algebras $(\mathfrak{n}, \langle \cdot, \cdot \rangle)$, defined by strengthening the cubic condition \eqref{grayA_condition} to the requirement that its Ricci operator satisfies
\begin{equation}\label{eq:intro_qc}
    [\operatorname{Ric} X, X] = 0 \quad \text{for all } X \in \mathfrak{n}.
\end{equation}
By polarization, \eqref{eq:intro_qc} is equivalent to $[\operatorname{Ric} X, Y] = [X, \operatorname{Ric} Y]$ for all $X, Y \in \mathfrak{n}$, which means precisely that the Ricci operator belongs to the \emph{quasicentroid} $\operatorname{QC}(\mathfrak{n})$ of the Lie algebra. Far from being an artificial algebraic restriction, the condition $\operatorname{Ric} \in \operatorname{QC}(\mathfrak{n})$ captures the exact operator-theoretic mechanism underlying known constructions of nilpotent $\mathcal{A}$-metrics, while providing a natural framework to unlock powerful structural decompositions.

To lay the algebraic foundations of our approach, after reviewing preliminary definitions in Section 2, we devote Section 3 to a self-contained study of the quasicentroid $\operatorname{QC}(\mathfrak{g})$ of an arbitrary Lie algebra $\mathfrak{g}$ with center $\mathfrak{z}$. In Proposition \ref{prop:qc_diagonalizable_iff}, we prove that a diagonalizable operator $Q \in L(\mathfrak{g}, \mathfrak{g})$ belongs to $\operatorname{QC}(\mathfrak{g})$ if and only if $\mathfrak{g}$ decomposes as a \textit{central product of ideals}, a concept adapted from group theory that acquires an interesting geometric interpretation in our study of $\mathcal{A}$-metrics. Furthermore, we establish a complete characterization of Lie algebras having a minimal quasicentroid $\operatorname{QC}(\mathfrak{g}) = \mathbb{R}\operatorname{Id}_{\mathfrak{g}} \oplus L(\mathfrak{g}, \mathfrak{z})$ in terms of \textit{central indecomposability} and the absence of complex or square-zero structures in $\operatorname{Sym}(\omega)$ (Theorem \ref{thm:qc_minimal_characterization}). Combined with auxiliary structural results, this characterization allows us to prove that every filiform Lie algebra has a minimal quasicentroid.

In Section 4, we apply this algebraic machinery to investigate the structural implications of the condition $\operatorname{Ric} \in \operatorname{QC}(\mathfrak{n})$ on a nilpotent Riemannian Lie algebra $(\mathfrak{n}, \langle \cdot, \cdot \rangle)$ with center $\mathfrak{z}$ and no abelian factors. We show that $\operatorname{Ric}$ is necessarily invertible and its eigenspaces split orthogonally between the center and its complement. Specifically, letting $E_{\lambda}$ denote the eigenspace associated with an eigenvalue $\lambda$, we obtain the decomposition $\mathfrak{z}^\perp = \bigoplus_{\lambda < 0} E_{\lambda}$ and $\mathfrak{z} = \bigoplus_{\lambda > 0} E_{\lambda}$. This leads to our main structural result, the Central Decomposition Theorem (Theorem \ref{thm:ricci_restriction}). For each negative eigenvalue $\lambda_j$ of $\operatorname{Ric}$, the subspace $\mathfrak{i}_j = E_{\lambda_j} \oplus \mathfrak{z}$ endowed with the restricted metric is an ideal of $\mathfrak{n}$ with center $\mathfrak{z}$, allowing $\mathfrak{n}$ to decompose as the central product of $\mathfrak{i}_1, \dots, \mathfrak{i}_r$. Furthermore, the intrinsic Ricci operator of each ideal satisfies $\operatorname{Ric}_{\mathfrak{i}_j} = \lambda_j \operatorname{Id}_{\mathfrak{i}_j} + \varphi_j \in \operatorname{QC}(\mathfrak{i}_j)$, where $\varphi_j$ is a linear map taking values in the center $\mathfrak{z}$. Geometrically, this represents the closest nilpotent analogue to a de Rham decomposition, successfully splitting the space into component factors that each inherently carry a Gray $\mathcal{A}$-metric.

Furthermore, we leverage this structural machinery to address specific classification problems. We prove that if $\dim \mathfrak{z} = 1$ and $\operatorname{Ric} \in \operatorname{QC}(\mathfrak{n})$, then $\mathfrak{n}$ must be two-step nilpotent and, consequently, isomorphic to an odd-dimensional Heisenberg algebra, a remarkable space where every inner product automatically defines a Gray $\mathcal{A}$-metric. To demonstrate that the hypothesis $\dim \mathfrak{z} = 1$ in Theorem \ref{thm:center_one_heisenberg} is sharp, we give metrics satisfying $\operatorname{Ric} \in \operatorname{QC}(\mathfrak{n})$ on the $5$-dimensional free $3$-step nilpotent Lie algebra and on a $7$-dimensional $5$-step nilpotent Lie algebra, both of which feature a two-dimensional center.

Finally, Section 5 addresses the converse question: under what conditions does a Gray $\mathcal{A}$-metric on a nilpotent Lie algebra necessarily satisfy $\operatorname{Ric} \in \operatorname{QC}(\mathfrak{n})$? We show that for any two-step nilpotent Riemannian Lie algebra, the $\mathcal{A}$-metric condition is equivalent to $\operatorname{Ric} \in \operatorname{QC}(\mathfrak{n})$, providing a unified, global reformulation of the restricted commutativity criteria established in \cite{AriasMarcoFB2026, ChenChenZhangZhu2025}. Furthermore, we prove that if a nilpotent Riemannian Lie algebra endowed with a Gray $\mathcal{A}$-metric admits a \emph{nice basis} (in the sense of Nikolayevsky \cite{Nikolayevsky2011}), its Ricci operator automatically satisfies the quasicentroid condition. We conclude the paper in Subsection 5.3 by removing all basis and step-length restrictions. Combining the polarized Gray identity with the vanishing of the Killing form, we prove in Theorem \ref{thm:grayA_trace} that a left-invariant metric on an arbitrary nilpotent Lie algebra satisfies $\operatorname{Ric} \in \operatorname{QC}(\mathfrak{n})$ if and only if it is a Gray $\mathcal{A}$-metric and the two symmetric bilinear forms
\begin{equation*}
    \theta_1(X,Y) = \operatorname{tr}(\operatorname{ad}_X  \operatorname{ad}_Y  \operatorname{Ric}) \quad \text{and} \quad \theta_2(X,Y) = \operatorname{tr}(\operatorname{ad}_X  \operatorname{Ric}  \operatorname{ad}_Y  \operatorname{Ric})
\end{equation*}
vanish identically.

\section{Preliminaries}

\begin{definition}
A Riemannian manifold $(M, g)$ is called an $\mathcal{A}$-\textit{manifold} (or is said to admit an $\mathcal{A}$-\textit{metric}) if its Ricci tensor $\operatorname{ric}$ is cyclic parallel, that is,
\begin{equation}\label{eq1}
(\nabla_{X}\operatorname{ric})(Y,Z) + (\nabla_{Y}\operatorname{ric})(Z,X) + (\nabla_{Z}\operatorname{ric})(X,Y) = 0
\end{equation}
for all vector fields $X,Y,Z \in \mathfrak{X}(M)$. Since the covariant derivative $\nabla \operatorname{ric}$ is symmetric in its last two arguments, standard polarization arguments show that condition \eqref{eq1} is equivalent to
\begin{equation}\label{eq2}
(\nabla_{X}\operatorname{ric})(X,X) = 0 \quad \text{for all } X \in \mathfrak{X}(M).
\end{equation}
In terms of the Ricci operator $\operatorname{Ric}$, defined by $g(\operatorname{Ric} X, Y) = \operatorname{ric}(X,Y)$, this can be equivalently written as
\begin{equation}\label{eq3}
g( (\nabla_{X}\operatorname{Ric}) X , X) = 0 \quad \text{for all } X \in \mathfrak{X}(M),
\end{equation}
where $(\nabla_{X}\operatorname{Ric})Y = \nabla_{X}(\operatorname{Ric}Y) - \operatorname{Ric}(\nabla_{X} Y)$.
\end{definition}

Let $G$ be a Lie group and $\mathfrak{g}$ its Lie algebra, which is the space of left-invariant vector fields on $G$. A \textit{left-invariant Riemannian metric} $g$ on $G$ is uniquely determined by its evaluation at the identity, which induces an inner product $\langle \cdot , \cdot \rangle$ on $\mathfrak{g}$ given by
\begin{equation*}
\langle X , Y \rangle = g(X,Y) \quad \text{for all } X,Y \in \mathfrak{g}.
\end{equation*}
The pair $(\mathfrak{g}, \langle \cdot , \cdot \rangle)$ is referred to as a \textit{Riemannian Lie algebra}. Conversely, any inner product on $\mathfrak{g}$ uniquely defines a left-invariant Riemannian metric on $G$.

For left-invariant vector fields $X, Y, Z \in \mathfrak{g}$, the Levi-Civita connection $\nabla$ of $(G,g)$ is determined by the Koszul formula:
\begin{equation}\label{koszul}
\langle \nabla_{X}Y , Z \rangle = \frac{1}{2}\big( \langle[X,Y],Z\rangle - \langle [Y,Z],X \rangle + \langle[Z,X],Y\rangle \big).
\end{equation}

We are interested in studying left-invariant $\mathcal{A}$-metrics on $G$. Since left-invariant vector fields globally frame the tangent bundle $TG$, and the function $g(X,Y)$ is constant on $G$ for all $X, Y \in \mathfrak{g}$, the geometric condition \eqref{eq3} translates into a purely algebraic property on the Riemannian Lie algebra. Indeed, as noted in \cite[Theorem 2.1]{ChenChenZhu2026}, a direct application of \eqref{koszul} shows that $g$ is an $\mathcal{A}$-metric if and only if
\begin{equation}\label{grayA_condition}
\langle [\operatorname{Ric} X , X] , X \rangle = 0 \quad \text{for all } X \in \mathfrak{g}.
\end{equation}

Condition \eqref{grayA_condition} naturally motivates the study of the \textit{stronger} algebraic condition:
\begin{equation}\label{eq5}
[\operatorname{Ric} X , X] = 0 \quad \text{for all } X \in \mathfrak{g}.
\end{equation}
It is clear that any left-invariant metric satisfying \eqref{eq5} is automatically an $\mathcal{A}$-metric. Although \textit{a priori} the geometric orthogonality condition \eqref{grayA_condition} does not necessarily force the purely algebraic identity \eqref{eq5}, assuming the latter defines a highly distinguished and algebraically rigid subclass of $\mathcal{A}$-metrics.

Unlike \eqref{eq3}, condition \eqref{eq5} cannot be meaningfully formulated for arbitrary Riemannian manifolds, as it relies on the Lie bracket rather than a standard geometric tensor. However, in the context of left-invariant metrics, this apparent lack of tensorial generality becomes a profound structural advantage: it reduces the problem entirely to the representation theory of the Lie algebra. By standard polarization arguments, equation \eqref{eq5} is equivalent to
\begin{equation}\label{eq6}
[\operatorname{Ric} X , Y] = [X , \operatorname{Ric}Y] \quad \text{for all } X,Y \in \mathfrak{g};
\end{equation}
which precisely means that $\operatorname{Ric}$ belongs to the \textit{quasicentroid} of the Lie algebra $\mathfrak{g}$. This motivates a brief algebraic discussion on this space of \textit{generalized derivations}, setting the stage to study the profound structural consequences of imposing that the Ricci operator lies in the quasicentroid.

\section{Interlude: Quasicentroids of Lie algebras}
This section reviews the notion of the quasicentroid of a Lie algebra and develops its core structural properties. The methods employed here are elementary, relying primarily on basic linear algebra techniques and adapting ideas provided in \cite{LegerLuks, FernandezCulma1, FernandezCulma2}. For the purposes of this algebraic section, and motivated by our subsequent geometric applications, all Lie algebras are assumed to be finite-dimensional over $\mathbb{R}$.

\begin{definition}
Let $\mathfrak{g}$ be a Lie algebra. A linear transformation $Q \in L(\mathfrak{g}, \mathfrak{g})$ is said to belong to the \textit{quasicentroid} of $\mathfrak{g}$ if it satisfies
\begin{equation}\label{eq:QC_def}
[Q X , Y] = [X , QY] \quad \text{for all } X,Y \in \mathfrak{g}.
\end{equation}
The vector space consisting of all such linear transformations is called the \textit{quasicentroid} of $\mathfrak{g}$ and is denoted by $\operatorname{QC}(\mathfrak{g})$.
\end{definition}

The space $\operatorname{QC}(\mathfrak{g})$ plays a distinguished role in the study of \textit{generalized derivations} of Lie algebras and was extensively investigated by Leger and Luks \cite[Section 5]{LegerLuks}, particularly in the context of Lie algebras with trivial center. In general, $\operatorname{QC}(\mathfrak{g})$ is not a Lie subalgebra of $\mathfrak{gl}(\mathfrak{g})$ under the standard commutator. Indeed, the commutator $D$ of any two elements in the quasicentroid is a \textit{$(0,1,1)$-derivation} of $\mathfrak{g}$, meaning it satisfies $[DX, Y] + [X, DY] = 0$ for all $X, Y \in \mathfrak{g}$. However, $\operatorname{QC}(\mathfrak{g})$ is closed under the \textit{standard symmetric product} of linear transformations, thereby forming a Jordan algebra (see \cite[Subsection 5.1]{LegerLuks}). Furthermore, if $\mathfrak{g}$ is an absolutely simple real Lie algebra, that is, its complexification $\mathfrak{g}_{\mathbb{C}}$ is simple, then Schur's Lemma, together with \cite[Theorem 5.28]{LegerLuks}, implies that $\operatorname{QC}(\mathfrak{g})=\mathbb{R}\operatorname{Id}_{\mathfrak{g}}$.

For Lie algebras with non-trivial center $\mathfrak{z}$, the quasicentroid is never reduced to scalar multiples of the identity. Let $L(\mathfrak{g}, \mathfrak{z})$ denote the vector space of all linear transformations from $\mathfrak{g}$ into $\mathfrak{z}$. A straightforward calculation shows that for any scalar $\lambda \in \mathbb{R}$ and any $\varphi \in L(\mathfrak{g}, \mathfrak{z})$, the operator $T = \lambda \operatorname{Id}_{\mathfrak{g}} + \varphi$ satisfies $[TX, Y] = \lambda[X, Y] = [X, TY]$, and hence belongs to $\operatorname{QC}(\mathfrak{g})$. Consequently, if $\mathfrak{g}$ is non-abelian, we always have the inclusion
\begin{equation}\label{eq:minimal_inclusion}
\mathbb{R}\operatorname{Id}_{\mathfrak{g}} \oplus L(\mathfrak{g}, \mathfrak{z}) \subseteq \operatorname{QC}(\mathfrak{g}).
\end{equation}
For convenience, when the equality holds in \eqref{eq:minimal_inclusion}, we shall say that $\mathfrak{g}$ has a \textit{minimal quasicentroid}.

In the classical structure theory of Lie algebras, a standard way to study the decomposability of a Lie algebra $\mathfrak{g}$ is through splittings into a direct sum of proper ideals $\mathfrak{g} = \mathfrak{i}_1 \oplus \mathfrak{i}_2$. For nilpotent Lie algebras, however, direct sum decompositions are overly restrictive because every non-zero ideal intersects the center $\mathfrak{z}$ non-trivially (in particular, any nilpotent Lie algebra with a one-dimensional center is indecomposable as a direct sum of ideals). To overcome this limitation when $\mathfrak{z} \neq \{0\}$, the natural generalization is the notion of a \textit{central product}. Originally rooted in group theory for the study of $p$-groups, this concept is naturally adapted to the context of Lie algebras (see, for instance, \cite[Definition 2.1]{Niroomand2016}).

\begin{definition}\label{def:centrally_decomposable}
Let $\mathfrak{g}$ be a Lie algebra with center $\mathfrak{z}$.
\begin{enumerate}
    \item[(i)] $\mathfrak{g}$ is said to be the \textit{central product} of a finite family of ideals $\mathfrak{i}_1, \dots, \mathfrak{i}_k$ if $\mathfrak{g} = \mathfrak{i}_1 + \dots + \mathfrak{i}_k$ and $[\mathfrak{i}_r, \mathfrak{i}_s] = \{0\}$ for all $r \neq s$.
    \item[(ii)] $\mathfrak{g}$ is said to be \textit{centrally decomposable} if it can be written as a central product of two proper ideals strictly containing the center; that is,
    \begin{equation*}
    \mathfrak{g} = \mathfrak{i}_1 + \mathfrak{i}_2, \quad [\mathfrak{i}_1, \mathfrak{i}_2] = \{0\}, \quad \text{and} \quad \mathfrak{z} \subsetneq \mathfrak{i}_1, \mathfrak{i}_2 \subsetneq \mathfrak{g}.
    \end{equation*}
    Otherwise, $\mathfrak{g}$ is said to be \textit{centrally indecomposable}.
\end{enumerate}
\end{definition}

\begin{remark}
Note that the mutual commutativity $[\mathfrak{i}_r, \mathfrak{i}_s] = \{0\}$ automatically forces the intersection of any ideal with the sum of the others to lie inside the center $\mathfrak{z}$. In particular, when $\mathfrak{i}_j \cap \sum_{r \neq j} \mathfrak{i}_r = \{0\}$ for all $j$, the central product of these ideals reduces to the standard direct sum $\mathfrak{i}_1 \oplus \dots \oplus \mathfrak{i}_k$. The strict inclusions $\mathfrak{z} \subsetneq \mathfrak{i}_k \subsetneq \mathfrak{g}$ in part (ii) rule out trivial decompositions (such as $\mathfrak{g} = \mathfrak{g} + \mathfrak{z}$), ensure that both ideals are non-abelian, and guarantee that their intersection coincides precisely with the center of the algebra ($\mathfrak{i}_1 \cap \mathfrak{i}_2 = \mathfrak{z}$).

A prototypical example of a centrally decomposable nilpotent Lie algebra is the $(2n+1)$-dimensional Heisenberg algebra $\mathfrak{h}_{2n+1} = \operatorname{span}\{x_1, y_1, \dots, x_n, y_n, z\}$, with non-zero brackets $[x_k, y_k] = z$, for $n \geq 2$. Although $\mathfrak{h}_{2n+1}$ is indecomposable as a direct sum of ideals because its center $\mathbb{R}z$ is one-dimensional, it decomposes as the central product of the ideals $\mathfrak{i}_1 = \operatorname{span}\{x_1, y_1, z\} \cong \mathfrak{h}_3$ and $\mathfrak{i}_2 = \operatorname{span}\{x_2, y_2, \dots, x_n, y_n, z\} \cong \mathfrak{h}_{2n-1}$, which commute pairwise and intersect precisely along $\mathfrak{z} = \mathbb{R}z$. By contrast, the $3$-dimensional Heisenberg algebra $\mathfrak{h}_3$ is centrally indecomposable.
\end{remark}

The assumption that a diagonalizable operator belongs to the quasicentroid can be elegantly translated into a pairwise commutativity condition on its eigenspaces.

\begin{proposition}\label{prop:qc_diagonalizable_iff}
Let $\mathfrak{g}$ be a Lie algebra with center $\mathfrak{z}$, and let $Q \in L(\mathfrak{g}, \mathfrak{g})$ be a diagonalizable operator with distinct real eigenvalues $\lambda_1, \dots, \lambda_s$ and corresponding eigenspaces $E_{\lambda_1}, \dots, E_{\lambda_s}$. For each $1 \le k \le s$, define the subspace $\mathfrak{i}_k = E_{\lambda_k} + \mathfrak{z}$. Then the following statements are equivalent:
\begin{enumerate}
    \item[(i)] $Q \in \operatorname{QC}(\mathfrak{g})$;
    \item[(ii)] $[E_{\lambda_i}, E_{\lambda_j}] = \{0\}$ for all $i \neq j$;
    \item[(iii)] each $\mathfrak{i}_k$ is an ideal of $\mathfrak{g}$, and $\mathfrak{g}$ decomposes as the central product of the ideals $\mathfrak{i}_1, \dots, \mathfrak{i}_s$.
\end{enumerate}
\end{proposition}

\begin{proof}
We first prove the equivalence between (i) and (ii). Suppose $Q \in \operatorname{QC}(\mathfrak{g})$. For any $i \neq j$, taking $X \in E_{\lambda_i}$ and $Y \in E_{\lambda_j}$, the quasicentroid condition yields
\begin{equation*}
    \lambda_i [X, Y] = [QX, Y] = [X, QY] = \lambda_j [X, Y].
\end{equation*}
Since $\lambda_i \neq \lambda_j$, we must have $[X, Y] = 0$, proving that $[E_{\lambda_i}, E_{\lambda_j}] = \{0\}$. Conversely, suppose that distinct eigenspaces commute pairwise. Let $X, Y \in \mathfrak{g}$ and decompose them as $X = \sum_{i=1}^s X_i$ and $Y = \sum_{j=1}^s Y_j$, where $X_i, Y_i \in E_{\lambda_i}$. Then
\begin{equation*}
    [QX, Y] = \sum_{i,j=1}^s [\lambda_i X_i, Y_j] = \sum_{i=1}^s \lambda_i [X_i, Y_i],
\end{equation*}
where the cross terms vanish by hypothesis. An identical calculation shows that $[X, QY] = \sum_{i=1}^s \lambda_i [X_i, Y_i]$, which implies $Q \in \operatorname{QC}(\mathfrak{g})$.

Next, we show that (ii) implies (iii). Assume that $[E_{\lambda_i}, E_{\lambda_j}] = \{0\}$ for all $i \neq j$. We first verify that $[E_{\lambda_k}, E_{\lambda_k}] \subseteq E_{\lambda_k} + \mathfrak{z}$ for each $1 \le k \le s$. Let $X, Y \in E_{\lambda_k}$ and write $W = [X, Y] = \sum_{m=1}^s W_m$ with $W_m \in E_{\lambda_m}$. For any index $j \neq k$ and any $Z \in E_{\lambda_j}$, since $X$ and $Y$ commute with $Z$ by (ii), the Jacobi identity gives
\begin{equation*}
    [W, Z] = [[X, Y], Z] = -[[Y, Z], X] - [[Z, X], Y] = 0.
\end{equation*}
On the other hand, expanding $[W, Z]$ and using (ii) again yields $[W, Z] = \sum_{m=1}^s [W_m, Z] = [W_j, Z]$. Thus $[W_j, Z] = 0$ for all $Z \in E_{\lambda_j}$. Since $W_j \in E_{\lambda_j}$ also commutes with $E_{\lambda_m}$ for all $m \neq j$ by (ii), it follows that $W_j$ commutes with all of $\mathfrak{g} = \bigoplus_{m=1}^s E_{\lambda_m}$, so $W_j \in \mathfrak{z}$ for every $j \neq k$. Consequently, $W = W_k + \sum_{j \neq k} W_j \in E_{\lambda_k} + \mathfrak{z} = \mathfrak{i}_k$, which proves that $[E_{\lambda_k}, E_{\lambda_k}] \subseteq \mathfrak{i}_k$.

Now, given any $X = \sum_{j=1}^s X_j \in \mathfrak{g}$ with $X_j \in E_{\lambda_j}$ and any $Y = Y_k + Z_0 \in \mathfrak{i}_k$ with $Y_k \in E_{\lambda_k}$ and $Z_0 \in \mathfrak{z}$, we have
\begin{equation*}
    [X, Y] = \sum_{j=1}^s [X_j, Y_k + Z_0] = [X_k, Y_k] \in [E_{\lambda_k}, E_{\lambda_k}] \subseteq \mathfrak{i}_k,
\end{equation*}
showing that each $\mathfrak{i}_k$ is an ideal of $\mathfrak{g}$. Moreover, $\mathfrak{g} = \sum_{k=1}^s E_{\lambda_k} = \sum_{k=1}^s \mathfrak{i}_k$, and for any $i \neq j$ we have $[\mathfrak{i}_i, \mathfrak{i}_j] = [E_{\lambda_i} + \mathfrak{z}, E_{\lambda_j} + \mathfrak{z}] = [E_{\lambda_i}, E_{\lambda_j}] = \{0\}$. Thus $\mathfrak{g}$ is the central product of the ideals $\mathfrak{i}_1, \dots, \mathfrak{i}_s$.

Finally, to see that (iii) implies (ii), note that if $\mathfrak{g}$ is the central product of $\mathfrak{i}_1, \dots, \mathfrak{i}_s$, then $[\mathfrak{i}_i, \mathfrak{i}_j] = \{0\}$ for all $i \neq j$ by Definition \ref{def:centrally_decomposable} (i). Since $E_{\lambda_i} \subseteq \mathfrak{i}_i$ and $E_{\lambda_j} \subseteq \mathfrak{i}_j$, it follows immediately that $[E_{\lambda_i}, E_{\lambda_j}] = \{0\}$ for all $i \neq j$.
\end{proof}

\begin{theorem}\label{thm:qc_central_decomposition}
Let $\mathfrak{g}$ be a non-abelian Lie algebra with center $\mathfrak{z}$. Then $\mathfrak{g}$ is centrally indecomposable if and only if every diagonalizable operator $Q \in \operatorname{QC}(\mathfrak{g})$ belongs to $\mathbb{R}\operatorname{Id}_{\mathfrak{g}} \oplus L(\mathfrak{g}, \mathfrak{z})$.
\end{theorem}

\begin{proof}
We prove the contrapositive of both implications. First, suppose that there exists a diagonalizable operator $Q \in \operatorname{QC}(\mathfrak{g})$ such that $Q \notin \mathbb{R}\operatorname{Id}_{\mathfrak{g}} \oplus L(\mathfrak{g}, \mathfrak{z})$, and let $E_{\lambda_1}, \dots, E_{\lambda_s}$ be the eigenspaces corresponding to the distinct real eigenvalues $\lambda_1, \dots, \lambda_s$ of $Q$. Notice that for any $1 \le m \le s$, the operator $Q - \lambda_m \operatorname{Id}_{\mathfrak{g}}$ acts on $E_{\lambda_k}$ by scalar multiplication by $(\lambda_k - \lambda_m)$, meaning its image is precisely $\bigoplus_{k \neq m} E_{\lambda_k}$. By hypothesis, $Q \notin \mathbb{R}\operatorname{Id}_{\mathfrak{g}} \oplus L(\mathfrak{g}, \mathfrak{z})$, so for any chosen eigenvalue $\lambda_m$, the operator $Q - \lambda_m \operatorname{Id}_{\mathfrak{g}}$ cannot map $\mathfrak{g}$ entirely into $\mathfrak{z}$. Consequently, the sum $\bigoplus_{k \neq m} E_{\lambda_k}$ is never contained in $\mathfrak{z}$. If at most one eigenspace were not contained in $\mathfrak{z}$ (say, $E_{\lambda_1} \not\subseteq \mathfrak{z}$ while $E_{\lambda_k} \subseteq \mathfrak{z}$ for all $k \neq 1$), then setting $m = 1$ would yield $\operatorname{Im}(Q - \lambda_1 \operatorname{Id}_{\mathfrak{g}}) = \bigoplus_{k \neq 1} E_{\lambda_k} \subseteq \mathfrak{z}$, yielding a contradiction. Therefore, at least two distinct eigenspaces, say $E_{\lambda_1}$ and $E_{\lambda_2}$, are not contained in $\mathfrak{z}$ (which, in particular, forces $s \ge 2$). By Proposition \ref{prop:qc_diagonalizable_iff}(iii), each $\mathfrak{i}_k = E_{\lambda_k} + \mathfrak{z}$ is an ideal of $\mathfrak{g}$ and $\mathfrak{g}$ is the central product of $\mathfrak{i}_1, \dots, \mathfrak{i}_s$. Defining the ideals
\begin{equation*}
    \mathfrak{j}_1 = \mathfrak{i}_1 = E_{\lambda_1} + \mathfrak{z} \quad \text{and} \quad \mathfrak{j}_2 = \sum_{k=2}^s \mathfrak{i}_k = \left(\bigoplus_{k=2}^s E_{\lambda_k}\right) + \mathfrak{z},
\end{equation*}
we have $\mathfrak{g} = \mathfrak{j}_1 + \mathfrak{j}_2$ and $[\mathfrak{j}_1, \mathfrak{j}_2] = \{0\}$. Because $E_{\lambda_1}, E_{\lambda_2} \not\subseteq \mathfrak{z}$ and $\mathfrak{g} = \bigoplus_{k=1}^s E_{\lambda_k}$, it follows that $\mathfrak{z} \subsetneq \mathfrak{j}_1, \mathfrak{j}_2 \subsetneq \mathfrak{g}$, so $\mathfrak{g}$ is centrally decomposable.

Conversely, suppose that $\mathfrak{g}$ is centrally decomposable, so that $\mathfrak{g} = \mathfrak{j}_1 + \mathfrak{j}_2$ for two ideals satisfying $[\mathfrak{j}_1, \mathfrak{j}_2] = \{0\}$ and $\mathfrak{z} \subsetneq \mathfrak{j}_1, \mathfrak{j}_2 \subsetneq \mathfrak{g}$. Since $[\mathfrak{j}_1, \mathfrak{j}_2] = \{0\}$, we have $\mathfrak{j}_1 \cap \mathfrak{j}_2 = \mathfrak{z}$. Let $\mathfrak{a}_1$ be a linear complement of $\mathfrak{z}$ in $\mathfrak{j}_1$; then $\mathfrak{a}_1 \neq \{0\}$, $\mathfrak{a}_1 \cap \mathfrak{z} = \{0\}$, and $\mathfrak{g} = \mathfrak{a}_1 \oplus \mathfrak{j}_2$. Define a diagonalizable operator $Q \in L(\mathfrak{g}, \mathfrak{g})$ with eigenspaces $E_1 = \mathfrak{a}_1$ and $E_0 = \mathfrak{j}_2$ by setting $Q|_{\mathfrak{a}_1} = \operatorname{Id}_{\mathfrak{a}_1}$ and $Q|_{\mathfrak{j}_2} = 0$. Since $[E_1, E_0] \subseteq [\mathfrak{j}_1, \mathfrak{j}_2] = \{0\}$, Proposition \ref{prop:qc_diagonalizable_iff}(ii) implies that $Q \in \operatorname{QC}(\mathfrak{g})$. Finally, we have that $Q \notin \mathbb{R}\operatorname{Id}_{\mathfrak{g}} \oplus L(\mathfrak{g}, \mathfrak{z})$. Suppose for a contradiction that $Q = \lambda \operatorname{Id}_{\mathfrak{g}} + \varphi$ for some $\lambda \in \mathbb{R}$ and $\varphi \in L(\mathfrak{g}, \mathfrak{z})$. Taking any non-zero $X \in \mathfrak{a}_1$, we have $X = QX = \lambda X + \varphi(X)$, which implies $(1-\lambda)X \in \mathfrak{z}$. Because $\mathfrak{a}_1 \cap \mathfrak{z} = \{0\}$, this forces $\lambda = 1$. On the other hand, since $\mathfrak{j}_2$ strictly contains $\mathfrak{z}$, we can choose an element $Y \in \mathfrak{j}_2 \setminus \mathfrak{z}$. Evaluating $Q$ on $Y$ with $\lambda = 1$ yields $0 = QY = Y + \varphi(Y)$, forcing $Y = -\varphi(Y) \in \mathfrak{z}$. This directly contradicts our choice of $Y$, allowing us to conclude that $Q \notin \mathbb{R}\operatorname{Id}_{\mathfrak{g}} \oplus L(\mathfrak{g}, \mathfrak{z})$.

\end{proof}

To study the quasicentroid $\operatorname{QC}(\mathfrak{g})$, it is natural to pass to the quotient $\mathfrak{g}/\mathfrak{z}$. Since the center $\mathfrak{z}$ is the \textit{radical} of the Lie bracket, the commutator induces a well-defined, non-degenerate, $[\mathfrak{g}, \mathfrak{g}]$-valued alternating bilinear form
\begin{equation*}
\omega\colon (\mathfrak{g}/\mathfrak{z}) \times (\mathfrak{g}/\mathfrak{z}) \to [\mathfrak{g}, \mathfrak{g}], \qquad \omega(\overline{X}, \overline{Y}) = [X, Y],
\end{equation*}
where non-degeneracy means that $\omega(\overline{X}, \overline{Y}) = 0$ for all $\overline{Y} \in \mathfrak{g}/\mathfrak{z}$ if and only if $\overline{X} = 0$. We denote by $\operatorname{Sym}(\omega)$ the subspace of $\omega$-symmetric operators on $\mathfrak{g}/\mathfrak{z}$:
\begin{equation*}
\operatorname{Sym}(\omega) = \big\{ A \in L(\mathfrak{g}/\mathfrak{z}, \mathfrak{g}/\mathfrak{z}) \;\big|\; \omega(A\overline{X}, \overline{Y}) = \omega(\overline{X}, A\overline{Y}) \text{ for all } \overline{X}, \overline{Y} \in \mathfrak{g}/\mathfrak{z} \big\}.
\end{equation*}

\begin{lemma}\label{qc_preserves_acs}
Let $\mathfrak{g}$ be a Lie algebra and let $Q \in \operatorname{QC}(\mathfrak{g})$. Then $Q$ leaves every term of the ascending central series of $\mathfrak{g}$ invariant. In particular, $Q(\mathfrak{z}) \subseteq \mathfrak{z}$, where $\mathfrak{z}$ is the center of $\mathfrak{g}$.
\end{lemma}

\begin{proof}
Recall that the ascending central series of $\mathfrak{g}$ is defined inductively by $\mathfrak{z}_0 = \{0\}$ and
\begin{equation*}
\mathfrak{z}_k = \{ Z \in \mathfrak{g} \mid [Z, X] \in \mathfrak{z}_{k-1} \text{ for all } X \in \mathfrak{g} \}.
\end{equation*}
We proceed by induction on $k \ge 0$. The base case $k=0$ is trivial. Assume that $Q(\mathfrak{z}_{k-1}) \subseteq \mathfrak{z}_{k-1}$ for some $k \ge 1$, and let $Z \in \mathfrak{z}_k$. For any $X \in \mathfrak{g}$, the quasicentroid condition yields
\begin{equation*}
[QZ, X] = [Z, QX].
\end{equation*}
Since $Z \in \mathfrak{z}_k$ and $QX \in \mathfrak{g}$, the bracket $[Z, QX]$ belongs to $\mathfrak{z}_{k-1}$ by definition. Thus, $[QZ, X] \in \mathfrak{z}_{k-1}$ for all $X \in \mathfrak{g}$, which implies $QZ \in \mathfrak{z}_k$. Therefore, $Q(\mathfrak{z}_k) \subseteq \mathfrak{z}_k$ for all $k \ge 0$. The invariance of the center $\mathfrak{z} = \mathfrak{z}_1$ corresponds to the case $k=1$.
\end{proof}

\begin{lemma}\label{lemma:qc_reduction_sym}
Let $\mathfrak{g}$ be a non-abelian Lie algebra with center $\mathfrak{z}$, and let $\pi\colon \mathfrak{g} \to \mathfrak{g}/\mathfrak{z}$ be the canonical projection. Then:
\begin{enumerate}
    \item[(i)] The map $\Phi\colon \operatorname{QC}(\mathfrak{g}) \to \operatorname{Sym}(\omega)$ given by $\Phi(Q) = \overline{Q}$, where $\overline{Q}(\overline{X}) = \overline{QX}$, is a surjective linear map with $\ker(\Phi) = L(\mathfrak{g}, \mathfrak{z})$. In particular, $\operatorname{QC}(\mathfrak{g}) / L(\mathfrak{g}, \mathfrak{z})$ and $\operatorname{Sym}(\omega)$ are isomorphic vector spaces. Also, $\operatorname{Sym}(\omega)$ is closed under powers: if $A \in \operatorname{Sym}(\omega)$, then $A^k \in \operatorname{Sym}(\omega)$ for all $k \ge 1$.
   \item[(ii)] $\mathfrak{g}$ is centrally decomposable if and only if the vector space $\mathfrak{g}/\mathfrak{z}$ admits a non-trivial $\omega$-orthogonal direct sum decomposition; i.e., there exist non-zero subspaces $\mathfrak{h}_1, \mathfrak{h}_2 \subseteq \mathfrak{g}/\mathfrak{z}$ such that $\mathfrak{g}/\mathfrak{z} = \mathfrak{h}_1 \oplus \mathfrak{h}_2$ and $\omega(\mathfrak{h}_1, \mathfrak{h}_2) = \{0\}$.
\end{enumerate}
\end{lemma}

\begin{proof}
(i) By Lemma \ref{qc_preserves_acs}, every $Q \in \operatorname{QC}(\mathfrak{g})$ preserves $\mathfrak{z}$ and induces $\overline{Q} \in \operatorname{Sym}(\omega)$, with $\overline{Q} = 0$ if and only if $Q \in L(\mathfrak{g}, \mathfrak{z})$. To see that $\Phi$ is surjective, note that the canonical projection $\pi\colon \mathfrak{g} \to \mathfrak{g}/\mathfrak{z}$ is a surjective linear map, so it admits a global linear section (a linear right inverse) $s\colon \mathfrak{g}/\mathfrak{z} \to \mathfrak{g}$ satisfying $\pi \circ s = \operatorname{Id}_{\mathfrak{g}/\mathfrak{z}}$. Given any $A \in \operatorname{Sym}(\omega)$, we define the lifted operator $Q \in L(\mathfrak{g}, \mathfrak{g})$ by the composition $Q = s \circ A \circ \pi$. For any $X, Y \in \mathfrak{g}$, we have
\begin{align*}
    [QX, Y] &= \omega(\pi(QX), \pi(Y)) && \text{(by definition of } \omega \text{)} \\
    &= \omega(A\overline{X}, \overline{Y}) && \text{(since } \pi \circ Q = A \circ \pi \text{)} \\
    &= \omega(\overline{X}, A\overline{Y}) && \text{(since } A \in \operatorname{Sym}(\omega) \text{)} \\
    &= \omega(\pi(X), \pi(QY)) && \text{(since } \pi \circ Q = A \circ \pi \text{)} \\
    &= [X, QY], && \text{(by definition of } \omega \text{)}
\end{align*}
so $Q \in \operatorname{QC}(\mathfrak{g})$. Moreover, the induced operator on the quotient satisfies $\overline{Q}(\overline{X}) = \pi(QX) = A \overline{X}$, meaning $\Phi(Q) = A$.

Finally, for any $A \in \operatorname{Sym}(\omega)$, induction on $k \ge 1$ yields
\begin{align*}
    \omega(A^k \overline{X}, \overline{Y}) &= \omega(A^{k-1}\overline{X}, A\overline{Y}) && \text{(since } A \in \operatorname{Sym}(\omega) \text{)} \\
    &= \omega(\overline{X}, A^k \overline{Y}), && \text{(by the inductive hypothesis)}
\end{align*}
so $A^k \in \operatorname{Sym}(\omega)$.

(ii) If $\mathfrak{g}$ is centrally decomposable as $\mathfrak{g} = \mathfrak{i}_1 + \mathfrak{i}_2$, then $\mathfrak{i}_1 \cap \mathfrak{i}_2 = \mathfrak{z}$. Consequently, the quotient spaces $\mathfrak{h}_1 = \mathfrak{i}_1/\mathfrak{z}$ and $\mathfrak{h}_2 = \mathfrak{i}_2/\mathfrak{z}$ form a direct sum $\mathfrak{g}/\mathfrak{z} = \mathfrak{h}_1 \oplus \mathfrak{h}_2$. Since $[\mathfrak{i}_1, \mathfrak{i}_2] = \{0\}$, this decomposition is $\omega$-orthogonal, and because $\mathfrak{z} \subsetneq \mathfrak{i}_k \subsetneq \mathfrak{g}$, both subspaces $\mathfrak{h}_1$ and $\mathfrak{h}_2$ are non-zero.

Conversely, suppose $\mathfrak{g}/\mathfrak{z} = \mathfrak{h}_1 \oplus \mathfrak{h}_2$ for some non-zero subspaces $\mathfrak{h}_1, \mathfrak{h}_2$ satisfying the condition $\omega(\mathfrak{h}_1, \mathfrak{h}_2) = \{0\}$. Let $\mathfrak{i}_1 = \pi^{-1}(\mathfrak{h}_1)$ and $\mathfrak{i}_2 = \pi^{-1}(\mathfrak{h}_2)$ be their preimages, so that $\mathfrak{g} = \mathfrak{i}_1 + \mathfrak{i}_2$ and $\mathfrak{i}_1 \cap \mathfrak{i}_2 = \mathfrak{z}$. The condition $\omega(\mathfrak{h}_1, \mathfrak{h}_2) = \{0\}$ means exactly that $[\mathfrak{i}_1, \mathfrak{i}_2] = \{0\}$. To see that $\mathfrak{i}_1$ is a subalgebra, let $X, Y \in \mathfrak{i}_1$ and write $[X, Y] = U + V$ with $U \in \mathfrak{i}_1$ and $V \in \mathfrak{i}_2$. For any $Z \in \mathfrak{i}_2$, the Jacobi identity and $[\mathfrak{i}_1, \mathfrak{i}_2] = \{0\}$ imply $[Z, [X, Y]] = 0$. Thus, $0 = [Z, U] + [Z, V] = [Z, V]$. This shows $V$ commutes with $\mathfrak{i}_2$, and since it trivially commutes with $\mathfrak{i}_1$, we get $V \in \mathfrak{z} \subset \mathfrak{i}_1$. Hence $[X,Y] \in \mathfrak{i}_1$, proving $[\mathfrak{i}_1, \mathfrak{i}_1] \subseteq \mathfrak{i}_1$. Finally, $[\mathfrak{g}, \mathfrak{i}_1] = [\mathfrak{i}_1 + \mathfrak{i}_2, \mathfrak{i}_1] = [\mathfrak{i}_1, \mathfrak{i}_1] \subseteq \mathfrak{i}_1$, meaning $\mathfrak{i}_1$ (and symmetrically $\mathfrak{i}_2$) is an ideal. This proves $\mathfrak{g}$ is centrally decomposable.
\end{proof}

To characterize when the entire quasicentroid $\operatorname{QC}(\mathfrak{g})$ reduces to the minimal subspace $\mathbb{R}\operatorname{Id}_{\mathfrak{g}} \oplus L(\mathfrak{g}, \mathfrak{z})$, we translate the problem to the quotient algebra $\overline{\mathfrak{g}} = \mathfrak{g}/\mathfrak{z}$. By Lemma \ref{lemma:qc_reduction_sym}(i), this is equivalent to showing that $\operatorname{Sym}(\omega) = \mathbb{R}\operatorname{Id}_{\overline{\mathfrak{g}}}$.

A way to understand how an operator $A \in \operatorname{Sym}(\omega)$ is forced to be a scalar multiple of the identity is through the real Jordan-Chevalley decomposition. Every operator $A$ on a finite-dimensional real vector space admits a unique decomposition $A = A_s + A_n$, where $A_s$ is semisimple (meaning every $A_s$-invariant subspace admits an $A_s$-invariant linear complement, which over $\mathbb{R}$ is equivalent to its minimal polynomial having no repeated factors), $A_n$ is nilpotent, and $A_s A_n = A_n A_s$. Crucially, both $A_s$ and $A_n$ are polynomials in $A$; therefore, since $\operatorname{Sym}(\omega)$ is closed under polynomials, any $A \in \operatorname{Sym}(\omega)$ guarantees $A_s, A_n \in \operatorname{Sym}(\omega)$.

To force $A$ to be a scalar multiple of the identity, we must structurally eliminate any deviation from $\mathbb{R}\operatorname{Id}_{\overline{\mathfrak{g}}}$:
\begin{enumerate}
   \item We must eliminate the nilpotent part ($A_n = 0$). Since $A_n \in \operatorname{Sym}(\omega)$, we must forbid the existence of non-trivial nilpotent operators in $\operatorname{Sym}(\omega)$. Note that forbidding nilpotent operators of arbitrary step $k \ge 2$ is equivalent to forbidding only 2-step nilpotent operators (i.e., square-zero operators).
    \item We must prevent the semisimple part $A_s$ from splitting the space. If its minimal polynomial had multiple distinct irreducible factors, we could construct a non-trivial idempotent operator that is a polynomial in $A_s$. This idempotent would belong to $\operatorname{Sym}(\omega)$, forcing a non-trivial $\omega$-orthogonal decomposition of $\overline{\mathfrak{g}}$, which we forbid by requiring $\mathfrak{g}$ to be centrally indecomposable.
    \item We must prevent the remaining single primary block of $A_s$ from being a rotation (arising from an irreducible quadratic polynomial), which we achieve by forbidding complex structures in $\operatorname{Sym}(\omega)$.
\end{enumerate}

These three structural constraints lead directly to the main characterization.

\begin{theorem}\label{thm:qc_minimal_characterization}
Let $\mathfrak{g}$ be a non-abelian Lie algebra with center $\mathfrak{z}$. Then $\operatorname{QC}(\mathfrak{g}) = \mathbb{R}\operatorname{Id}_{\mathfrak{g}} \oplus L(\mathfrak{g}, \mathfrak{z})$ if and only if the following three conditions hold:
\begin{enumerate}
    \item[(i)] $\mathfrak{g}$ is centrally indecomposable;
    \item[(ii)] $\operatorname{Sym}(\omega)$ contains no complex structure, i.e., there is no $J \in \operatorname{Sym}(\omega)$ such that $J^2 = -\operatorname{Id}_{\overline{\mathfrak{g}}}$;
    \item[(iii)] $\operatorname{Sym}(\omega)$ contains no non-zero square-zero operator, i.e., there is no $0 \neq N \in \operatorname{Sym}(\omega)$ such that $N^2 = 0$.
\end{enumerate}
\end{theorem}

\begin{proof}
Let $\overline{\mathfrak{g}} = \mathfrak{g}/\mathfrak{z}$. By Lemma \ref{lemma:qc_reduction_sym}(i), the condition $\operatorname{QC}(\mathfrak{g}) = \mathbb{R}\operatorname{Id}_{\mathfrak{g}} \oplus L(\mathfrak{g}, \mathfrak{z})$ is equivalent to $\operatorname{Sym}(\omega) = \mathbb{R}\operatorname{Id}_{\overline{\mathfrak{g}}}$.

Suppose first that $\operatorname{Sym}(\omega) = \mathbb{R}\operatorname{Id}_{\overline{\mathfrak{g}}}$. Then, conditions (ii) and (iii) are trivially satisfied. To verify condition (i), suppose $\overline{\mathfrak{g}}$ admits an $\omega$-orthogonal decomposition $\overline{\mathfrak{g}} = \mathfrak{h}_1 \oplus \mathfrak{h}_2$. Let $P$ be the projection onto $\mathfrak{h}_1$ along $\mathfrak{h}_2$, which satisfies $P^2 = P$. Because the decomposition is $\omega$-orthogonal, $P$ is symmetric with respect to $\omega$: for any $X = X_1 + X_2$ and $Y = Y_1 + Y_2$ (with $X_i, Y_i \in \mathfrak{h}_i$), the orthogonality $\omega(\mathfrak{h}_1, \mathfrak{h}_2) = \omega(\mathfrak{h}_2, \mathfrak{h}_1) = \{0\}$ yields
\begin{align*}
    \omega(PX, Y) &= \omega(X_1, Y_1 + Y_2) && \text{(by definition of } P \text{)} \\
    &= \omega(X_1, Y_1) && \text{(since } \omega(\mathfrak{h}_1, \mathfrak{h}_2) = \{0\} \text{)} \\
    &= \omega(X_1 + X_2, Y_1) && \text{(since } \omega(\mathfrak{h}_2, \mathfrak{h}_1) = \{0\} \text{)} \\
    &= \omega(X, PY), && \text{(by definition of } P \text{)}
\end{align*}
so $P \in \operatorname{Sym}(\omega)$. By our initial assumption, $P$ must be a scalar multiple of the identity, meaning $P = \lambda \operatorname{Id}_{\overline{\mathfrak{g}}}$. The idempotent condition $P^2 = P$ forces $\lambda^2 = \lambda$, so $\lambda \in \{0, 1\}$. This means either $P = 0$ or $P = \operatorname{Id}_{\overline{\mathfrak{g}}}$, which implies $\mathfrak{h}_1 = \{0\}$ or $\mathfrak{h}_2 = \{0\}$. Therefore, $\overline{\mathfrak{g}}$ admits no non-trivial $\omega$-orthogonal decompositions, and Lemma \ref{lemma:qc_reduction_sym}(ii) concludes that $\mathfrak{g}$ is centrally indecomposable.

Conversely, assume that conditions (i), (ii), and (iii) hold, and take any $A \in \operatorname{Sym}(\omega)$. Consider its real Jordan-Chevalley decomposition $A = A_s + A_n$ where $A_s$ is semisimple and  $A_n$ is nilpotent. Since $A_n$ is polynomial in $A$, we have $A_n \in \operatorname{Sym}(\omega)$. If $A_n \neq 0$, it has an index of nilpotence $k \ge 2$. Then the operator $N = A_n^{\lceil k/2 \rceil}$ is a non-zero polynomial in $A_n$, meaning $N \in \operatorname{Sym}(\omega)$, and satisfies $N^2 = 0$, which violates condition (iii). Thus, we must have $A_n = 0$, and so $A = A_s$ is semisimple.

Since $A$ is semisimple, its minimal polynomial $m_A(t)$ has no repeated factors. Suppose $m_A(t)$ has two or more distinct irreducible factors. We can write $m_A(t) = f(t)h(t)$, where $f(t)$ and $h(t)$ are relatively prime, non-constant polynomials. B\'ezout's identity provides polynomials $u(t)$ and $v(t)$ such that $u(t)f(t) + v(t)h(t) = 1$. Evaluating this relation at the operator $A$ yields $u(A)f(A) + v(A)h(A) = \operatorname{Id}_{\overline{\mathfrak{g}}}$. We define $P = u(A)f(A)$. Because $f(A)h(A) = m_A(A) = 0$, it follows that $P(\operatorname{Id}_{\overline{\mathfrak{g}}} - P) = u(A)v(A)f(A)h(A) = 0$, meaning $P^2 = P$. Since $f$ and $h$ are proper factors of the minimal polynomial, $P$ is a non-trivial idempotent ($P \neq 0$ and $P \neq \operatorname{Id}_{\overline{\mathfrak{g}}}$). Furthermore, because $P$ is a polynomial in $A$, it belongs to $\operatorname{Sym}(\omega)$. Any idempotent operator naturally decomposes the space as $\overline{\mathfrak{g}} = \operatorname{Im}(P) \oplus \ker(P)$, and the fact that $P \in  \operatorname{Sym}(\omega)$ ensures this decomposition is $\omega$-orthogonal: for any $x \in \operatorname{Im}(P)$ and $y \in \ker(P)$, we compute:
\begin{align*}
    \omega(x, y) &= \omega(Px, y) && \text{(since } Px = x \text{ for } x \in \operatorname{Im}(P) \text{)} \\
    &= \omega(x, Py) && \text{(since } P \in \operatorname{Sym}(\omega) \text{)} \\
    &= \omega(x, 0) = 0. && \text{(since } y \in \ker(P) \text{)}
\end{align*}
This $\omega$-orthogonal decomposition implies $\mathfrak{g}$ is centrally decomposable by Lemma \ref{lemma:qc_reduction_sym}(ii), violating condition (i).

Therefore, $m_A(t)$ must consist of a single irreducible factor over $\mathbb{R}$. If this factor were quadratic, then $m_A(t) = (t-\alpha)^2 + \beta^2$ with $\beta \neq 0$. The operator $J = \beta^{-1}(A - \alpha \operatorname{Id}_{\overline{\mathfrak{g}}})$ would be a polynomial in $A$, hence $J \in \operatorname{Sym}(\omega)$, and would satisfy $J^2 = -\operatorname{Id}_{\overline{\mathfrak{g}}}$, violating condition (ii).

The only remaining structural possibility is that $m_A(t) = t - \lambda$ for some $\lambda \in \mathbb{R}$. This immediately implies $A = \lambda \operatorname{Id}_{\overline{\mathfrak{g}}}$. Thus, $\operatorname{Sym}(\omega) = \mathbb{R}\operatorname{Id}_{\overline{\mathfrak{g}}}$.
\end{proof}

We next establish a general structural criterion that guarantees the minimality of the quasicentroid whenever the Lie algebra admits an element whose centralizer modulo the center is one-dimensional.

\begin{theorem}\label{thm:qc_centralizer_criterion}
Let $\mathfrak{g}$ be a non-abelian Lie algebra with center $\mathfrak{z}$. If there exists an element $X_0 \in \mathfrak{g} \setminus \mathfrak{z}$ whose centralizer $C_{\mathfrak{g}}(X_0) = \ker(\operatorname{ad}_{X_0})$ satisfies
\begin{equation*}
    C_{\mathfrak{g}}(X_0) = \mathbb{R}X_0 \oplus \mathfrak{z},
\end{equation*}
then the quasicentroid of $\mathfrak{g}$ is minimal; that is,
\begin{equation*}
    \operatorname{QC}(\mathfrak{g}) = \mathbb{R}\operatorname{Id}_{\mathfrak{g}} \oplus L(\mathfrak{g}, \mathfrak{z}).
\end{equation*}
\end{theorem}

\begin{proof}
To prove the general statement, recall from Lemma \ref{lemma:qc_reduction_sym}(i) that it suffices to show $\operatorname{Sym}(\omega) = \mathbb{R}\operatorname{Id}_{\overline{\mathfrak{g}}}$. Let $\bar{X}_0 \in \overline{\mathfrak{g}}$ denote the coset $X_0 + \mathfrak{z}$. Because $X_0 \notin \mathfrak{z}$, the vector $\bar{X}_0$ is non-zero. We first determine the kernel of the linear map $\omega(\bar{X}_0, \cdot)\colon \overline{\mathfrak{g}} \to \mathfrak{g}$. An element $\bar{Y} \in \overline{\mathfrak{g}}$ lies in this kernel if and only if $\omega(\bar{X}_0, \bar{Y}) = [X_0, Y] = 0$ in $\mathfrak{g}$, which means precisely that $Y \in C_{\mathfrak{g}}(X_0) = \mathbb{R}X_0 \oplus \mathfrak{z}$. Passing to the quotient space $\overline{\mathfrak{g}}$, this condition is equivalent to $\bar{Y} \in \mathbb{R}\bar{X}_0$. Consequently,
\begin{equation*}
    \ker(\omega(\bar{X}_0, \cdot)) = \mathbb{R}\bar{X}_0.
\end{equation*}

Now let $A \in \operatorname{Sym}(\omega)$ be an arbitrary operator. For every $\bar{Y} \in \overline{\mathfrak{g}}$, combining the $\omega$-symmetry of $A$ with the skew-symmetry of $\omega$ yields
\begin{equation*}
    \omega(A\bar{Y}, \bar{Y}) = \omega(\bar{Y}, A\bar{Y}) = -\omega(A\bar{Y}, \bar{Y}),
\end{equation*}
which implies that $\omega(A\bar{Y}, \bar{Y}) = 0$ for all $\bar{Y} \in \overline{\mathfrak{g}}$. Evaluating this identity at $\bar{Y} = \bar{X}_0$, we obtain $\omega(\bar{X}_0, A\bar{X}_0) = -\omega(A\bar{X}_0, \bar{X}_0) = 0$. Thus $A\bar{X}_0 \in \ker(\omega(\bar{X}_0, \cdot)) = \mathbb{R}\bar{X}_0$, so there exists a scalar $\lambda \in \mathbb{R}$ such that
\begin{equation*}
    A\bar{X}_0 = \lambda \bar{X}_0.
\end{equation*}

Define the linear operator $B = A - \lambda \operatorname{Id}_{\overline{\mathfrak{g}}}$. Since $\operatorname{Sym}(\omega)$ is a linear subspace containing the identity operator, we have $B \in \operatorname{Sym}(\omega)$ and, by construction, $B\bar{X}_0 = 0$. For any element $\bar{Y} \in \overline{\mathfrak{g}}$, the symmetry of $B$ with respect to $\omega$ gives
\begin{align*}
    \omega(\bar{X}_0, B\bar{Y}) &= \omega(B\bar{X}_0, \bar{Y}) && \text{(since } B \in \operatorname{Sym}(\omega) \text{)} \\
    &= \omega(0, \bar{Y}) = 0. && \text{(since } B\bar{X}_0 = 0 \text{)}
\end{align*}
Therefore, $B\bar{Y}$ belongs to $\ker(\omega(\bar{X}_0, \cdot)) = \mathbb{R}\bar{X}_0$ for every $\bar{Y} \in \overline{\mathfrak{g}}$. This guarantees the existence of a linear functional $\mu\colon \overline{\mathfrak{g}} \to \mathbb{R}$ such that
\begin{equation*}
    B\bar{Y} = \mu(\bar{Y})\bar{X}_0 \quad \text{for all } \bar{Y} \in \overline{\mathfrak{g}}.
\end{equation*}
Note also that because $B\bar{X}_0 = 0$ and $\bar{X}_0 \neq 0$, this functional satisfies $\mu(\bar{X}_0) = 0$.

Finally, we show that $\mu$ vanishes identically. Since $B \in \operatorname{Sym}(\omega)$, the identity $\omega(B\bar{Y}, \bar{Y}) = 0$ holds for every $\bar{Y} \in \overline{\mathfrak{g}}$. Substituting $B\bar{Y} = \mu(\bar{Y})\bar{X}_0$ into this relation, the bilinearity of $\omega$ yields
\begin{equation*}
    0 = \omega(\mu(\bar{Y})\bar{X}_0, \bar{Y}) = \mu(\bar{Y})\omega(\bar{X}_0, \bar{Y}).
\end{equation*}
For any element $\bar{Y} \in \overline{\mathfrak{g}}$ not lying in the one-dimensional subspace $\mathbb{R}\bar{X}_0 = \ker(\omega(\bar{X}_0, \cdot))$, the vector $\omega(\bar{X}_0, \bar{Y})$ is non-zero in $\mathfrak{g}$, which forces $\mu(\bar{Y}) = 0$. Since $\mu(\bar{X}_0) = 0$ also holds, the linearity of $\mu$ dictates that it must vanish identically on $\overline{\mathfrak{g}}$. Consequently, $B = 0$, which yields $A = \lambda \operatorname{Id}_{\overline{\mathfrak{g}}}$. This establishes that $\operatorname{Sym}(\omega) = \mathbb{R}\operatorname{Id}_{\overline{\mathfrak{g}}}$ and completes the proof.
\end{proof}

As a direct application of Theorem \ref{thm:qc_centralizer_criterion}, we can determine the quasicentroid of any filiform Lie algebra. To properly contextualize this class, recall that the lower central series of a Lie algebra $\mathfrak{g}$ is defined inductively by $\mathfrak{g}^0 = \mathfrak{g}$ and $\mathfrak{g}^k = [\mathfrak{g}, \mathfrak{g}^{k-1}]$ for all integers $k \ge 1$. The algebra $\mathfrak{g}$ is called \emph{nilpotent} if this series eventually vanishes. When the series vanishes exactly at the $s$-th term; that is, $\mathfrak{g}^s = \{0\}$ while $\mathfrak{g}^{s-1} \neq \{0\}$, we say that $s$ is the \emph{step of nilpotency} of $\mathfrak{g}$ (and refer to it as an $s$-step nilpotent Lie algebra). For any nilpotent Lie algebra, the dimension of the ideals $\mathfrak{g}^k$ must strictly decrease at each step until reaching zero. Consequently, the maximum possible nilpotency step for an $n$-dimensional nilpotent Lie algebra is $n-1$. Algebras attaining this extreme bound form a distinguished structural class.

\begin{definition}\label{def:filiform}
A nilpotent Lie algebra $\mathfrak{g}$ of dimension $n \ge 3$ is called \emph{filiform} if its step of nilpotency is $n-1$.
\end{definition}

As shown by Mich\`ele Vergne in her work on the variety of nilpotent Lie algebras \cite{Vergne1970}, every filiform Lie algebra admits a special basis compatible with its lower central series.

\begin{proposition}[\cite{Vergne1970}]
\label{prop:vergne_basis}
Let $\mathfrak{g}$ be a filiform Lie algebra of dimension $n \ge 3$ over $\mathbb{R}$. Then there exists a basis $\{e_1, e_2, \dots, e_n\}$ of $\mathfrak{g}$ satisfying the following properties:
\begin{enumerate}
    \item[(i)] $[e_1, e_i] = e_{i+1}$ for all $2 \le i \le n-1$, and $[e_1, e_n] = 0$;
    \item[(ii)] for each $1 \le k \le n-2$, the ideal $\mathfrak{g}^k$ in the lower central series is given by $$\mathfrak{g}^k = \operatorname{span}\{e_{k+2}, \dots, e_n\}.$$
\end{enumerate}
\end{proposition}

\begin{corollary}\label{cor:qc_filiform}
If $\mathfrak{g}$ is a filiform Lie algebra, then its quasicentroid is minimal; that is,
\begin{equation*}
    \operatorname{QC}(\mathfrak{g}) = \mathbb{R}\operatorname{Id}_{\mathfrak{g}} \oplus L(\mathfrak{g}, \mathfrak{z}).
\end{equation*}
\end{corollary}

\begin{proof}
Let $\{e_1, e_2, \dots, e_n\}$ be a basis of $\mathfrak{g}$ provided by Proposition \ref{prop:vergne_basis}. Given an arbitrary element $x = \sum_{i=1}^n c_i e_i \in \mathfrak{g}$, property (i) of that proposition yields
\begin{equation*}
    \operatorname{ad}_{e_1}(x) = [e_1, x] = \sum_{i=2}^{n-1} c_i e_{i+1}.
\end{equation*}
Because the vectors $e_3, \dots, e_n$ are linearly independent, we have $[e_1, x] = 0$ if and only if $c_2 = \dots = c_{n-1} = 0$. Consequently, the centralizer of $e_1$ is the two-dimensional subspace
\begin{equation*}
    C_{\mathfrak{g}}(e_1) = \ker(\operatorname{ad}_{e_1}) = \operatorname{span}\{e_1, e_n\}.
\end{equation*}
Furthermore, since filiform Lie algebras have a one-dimensional center and it is generated, in this presentation, by $e_n$, we obtain $C_{\mathfrak{g}}(e_1) = \mathbb{R}e_1 \oplus \mathfrak{z}$. Applying Theorem \ref{thm:qc_centralizer_criterion} with $X_0 = e_1$ yields the result.
\end{proof}

\section{Structural Implications of a Quasicentroid Ricci Operator}

In this section, we focus on studying the structure of nilpotent Riemannian Lie algebras $(\mathfrak{n}, \langle \cdot, \cdot \rangle)$ without abelian factors whose Ricci operator belongs to the quasicentroid.

As a first step to understand this structure, we analyze the behavior of $\operatorname{Ric}$ with respect to the center $\mathfrak{z}$ of $\mathfrak{n}$. By Lemma \ref{qc_preserves_acs}, the assumption $\operatorname{Ric} \in \operatorname{QC}(\mathfrak{n})$ implies that $\operatorname{Ric}(\mathfrak{z}) \subseteq \mathfrak{z}$. Because $\operatorname{Ric}$ is self-adjoint with respect to $\langle \cdot, \cdot \rangle$, it also preserves the orthogonal complement $\mathfrak{z}^\perp$. From basic linear algebra, this only guarantees that each eigenspace splits orthogonally as $E_{\lambda_j} = (E_{\lambda_j} \cap \mathfrak{z}^\perp) \oplus (E_{\lambda_j} \cap \mathfrak{z})$. However, using the classical Ricci formula for nilpotent Lie algebras, we can prove a much stronger geometric property: the eigenvalues of $\operatorname{Ric}$ on $\mathfrak{z}$ are strictly positive, while those on $\mathfrak{z}^\perp$ are strictly negative. In particular, each eigenspace of $\operatorname{Ric}$ is entirely contained either in $\mathfrak{z}$ or in $\mathfrak{z}^\perp$.

To establish this, recall that the Ricci operator of a general Riemannian Lie algebra $(\mathfrak{g}, \langle \cdot, \cdot \rangle)$ takes the form (see, e.g., \cite[Corollary 7.38]{Besse})
\begin{equation}\label{ricci}
    \operatorname{Ric} = R - \frac{1}{2}B - \operatorname{ad}_{H}^{\operatorname{sym}},
\end{equation}
where $R$ is the symmetric operator defined by
\begin{equation}\label{momentmap}
\langle R X , X\rangle = -\frac{1}{2}\sum_{i,j=1}^n \langle [X,e_i],e_j\rangle^{2} + \frac{1}{4}\sum_{i,j=1}^n \langle [e_i,e_j] ,X \rangle^{2} \quad \text{for all } X\in \mathfrak{g},
\end{equation}
for an arbitrary orthonormal basis $\{e_1, \dots, e_n\}$ of $\mathfrak{g}$, the linear map $B$ is the symmetric operator representing the Killing form of $\mathfrak{g}$ with respect to $\langle \cdot, \cdot \rangle$, that is,
\begin{equation*}
\langle B X,Y \rangle = \operatorname{tr}(\operatorname{ad}_X  \operatorname{ad}_Y),
\end{equation*}
and $\operatorname{ad}_{H}^{\operatorname{sym}}$ denotes the symmetric part of $\operatorname{ad}_{H}$, where $H \in \mathfrak{g}$ is the \textit{mean curvature vector} defined by
\begin{equation*}
\operatorname{tr}(\operatorname{ad}_X) = \langle H, X\rangle.
\end{equation*}

\begin{lemma}\label{lemma:ricci_eigenspace_splitting}
Let $(\mathfrak{n}, \langle \cdot, \cdot \rangle)$ be a nilpotent Riemannian Lie algebra with no abelian factor, let $\mathfrak{z}$ denote its center, and suppose that $\operatorname{Ric} \in \operatorname{QC}(\mathfrak{n})$. Let $\lambda_1, \dots, \lambda_s$ be the distinct real eigenvalues of $\operatorname{Ric}$, with corresponding eigenspaces $E_{\lambda_1}, \dots, E_{\lambda_s}$. Then:
\begin{enumerate}
    \item[(i)] If $E_{\lambda_j} \cap \mathfrak{z} \neq \{0\}$, then $\lambda_j > 0$.
    \item[(ii)] If $E_{\lambda_j} \cap \mathfrak{z}^\perp \neq \{0\}$, then $\lambda_j < 0$.
\end{enumerate}
Consequently, $\operatorname{Ric}$ is non-singular and each eigenspace $E_{\lambda_j}$ is contained either in $\mathfrak{z}^\perp$ or in $\mathfrak{z}$. Up to reordering the eigenvalues, there exists an integer $1 \le r < s$ such that $\lambda_j < 0$ for $1 \le j \le r$ and $\lambda_j > 0$ for $r+1 \le j \le s$, yielding the orthogonal direct sum decompositions
\begin{equation*}
    \mathfrak{z}^\perp = \bigoplus_{j=1}^r E_{\lambda_j} \quad \text{and} \quad \mathfrak{z} = \bigoplus_{j=r+1}^s E_{\lambda_j}.
\end{equation*}
\end{lemma}

\begin{proof}
Since $\mathfrak{n}$ is nilpotent, its Killing form $B$ vanishes and it is unimodular, so the mean curvature vector $H$ is zero. Thus, formula \eqref{ricci} reduces to $\operatorname{Ric} = R$ and for any orthonormal basis $\{e_1, \dots, e_n\}$ of $(\mathfrak{n}, \langle \cdot, \cdot \rangle)$ and any $X \in \mathfrak{n}$, we have
\begin{equation}\label{eq:ricci_formula}
    \langle \operatorname{Ric} X, X \rangle = -\frac{1}{2} \sum_{p,q=1}^n \langle [X, e_p], e_q \rangle^2 + \frac{1}{4} \sum_{p,q=1}^n \langle [e_p, e_q], X \rangle^2.
\end{equation}
As noted above, each eigenspace decomposes orthogonally as $E_{\lambda_j} = \mathfrak{a}_j \oplus \mathfrak{z}_j$, where $\mathfrak{a}_j = E_{\lambda_j} \cap \mathfrak{z}^\perp$ and $\mathfrak{z}_j = E_{\lambda_j} \cap
\mathfrak{z}$.

The positivity of the Ricci operator on the center, as stated in part (i), is a well-known fact for nilpotent Lie algebras without abelian factors. We briefly recall the argument for completeness. Suppose that $E_{\lambda_j} \cap \mathfrak{z} \neq \{0\}$ and choose a unit vector $Z \in E_{\lambda_j} \cap \mathfrak{z}$. Since $[Z, \mathfrak{n}] = \{0\}$, applying formula \eqref{eq:ricci_formula} to $Z$ with respect to an orthonormal basis $\{e_1, \dots, e_n\}$ of $\mathfrak{n}$ yields
\begin{equation*}
    \lambda_j = \langle \operatorname{Ric} Z, Z \rangle = \frac{1}{4} \sum_{p,q=1}^n \langle [e_p, e_q], Z \rangle^2 \ge 0.
\end{equation*}
Because $\mathfrak{n}$ has no abelian factor, its center is entirely contained in the commutator ideal, meaning $\mathfrak{z} \subseteq [\mathfrak{n}, \mathfrak{n}]$. Consequently, the non-zero vector $Z$ cannot be orthogonal to $[\mathfrak{n}, \mathfrak{n}]$, which forces at least one term in the sum to be strictly positive. This establishes that $\lambda_j > 0$.

To prove (ii), suppose that $d = \dim \mathfrak{a}_j \ge 1$. Since the eigenspaces of $\operatorname{Ric}$ are mutually orthogonal, we have the orthogonal direct sum $\mathfrak{n} = \mathfrak{a}_1 \oplus \dots \oplus \mathfrak{a}_s \oplus \mathfrak{z}$. Choose an orthonormal basis $\{e_1, \dots, e_n\}$ of $\mathfrak{n}$ adapted to this decomposition, and denote the basis elements spanning $\mathfrak{a}_j$ and $\mathfrak{z}$ by $\{a_1, \dots, a_d\}$ and $\{z_1, \dots, z_m\}$, respectively. By Proposition \ref{prop:qc_diagonalizable_iff}, $[E_{\lambda_j}, E_{\lambda_k}] = \{0\}$ for all $k \neq j$, and $[E_{\lambda_k}, E_{\lambda_k}] \subseteq E_{\lambda_k} + \mathfrak{z} = \mathfrak{a}_k \oplus \mathfrak{z}$ for every $k$. In particular, for any fixed $a_i \in \mathfrak{a}_j$:
\begin{enumerate}
    \item[(a)] The bracket $[a_i, e_p]$ vanishes for all $e_p \notin \mathfrak{a}_j$, and for $a_p \in \mathfrak{a}_j$ we have $[a_i, a_p] \in \mathfrak{a}_j \oplus \mathfrak{z}$.
    \item[(b)] If $e_p$ and $e_q$ belong to different eigenspaces, or if either belongs to $\mathfrak{z}$, their bracket is zero. Thus $[e_p, e_q]$ can only be non-zero when both $e_p, e_q \in \mathfrak{a}_k$ for some $k$, in which case $[e_p, e_q] \in \mathfrak{a}_k \oplus \mathfrak{z}$. By orthogonality, $\langle [e_p, e_q], a_i \rangle = 0$ unless $k = j$, that is, unless both $e_p, e_q \in \mathfrak{a}_j$.
\end{enumerate}
By (a) and (b), evaluating \eqref{eq:ricci_formula} on $X = a_i$ yields
\begin{equation*}
    \lambda_j = \langle \operatorname{Ric} a_i, a_i \rangle = -\frac{1}{2} \sum_{p,q=1}^d \langle [a_i, a_p], a_q \rangle^2 - \frac{1}{2} \sum_{p=1}^d \sum_{\alpha=1}^m \langle [a_i, a_p], z_\alpha \rangle^2 + \frac{1}{4} \sum_{p,q=1}^d \langle [a_p, a_q], a_i \rangle^2.
\end{equation*}
Summing this equation over $i = 1, \dots, d$, and using the fact that
\begin{equation*}
    \sum_{i,p,q=1}^d \langle [a_i, a_p], a_q \rangle^2 = \sum_{i,p,q=1}^d \langle [a_p, a_q], a_i \rangle^2,
\end{equation*}
we obtain
\begin{equation*}
    d\lambda_j = -\frac{1}{4} \sum_{i,p,q=1}^d \langle [a_i, a_p], a_q \rangle^2 - \frac{1}{2} \sum_{i,p=1}^d \sum_{\alpha=1}^m \langle [a_i, a_p], z_\alpha \rangle^2 \le 0.
\end{equation*}
If $d\lambda_j = 0$, both sums vanish, which implies $[a_i, a_p] = 0$ for all $1 \le i, p \le d$. Since $[a_i, E_{\lambda_k}] = \{0\}$ for all $k \neq j$, this gives $[a_i, \mathfrak{n}] = \{0\}$, so $a_i \in \mathfrak{z}$, contradicting $a_i \in \mathfrak{z}^\perp$. Therefore $d\lambda_j < 0$, which proves $\lambda_j < 0$.

Finally, parts (i) and (ii) immediately imply that we cannot have both $\mathfrak{z}_j \neq \{0\}$ and $\mathfrak{a}_j \neq \{0\}$ for the same eigenvalue $\lambda_j$. Hence either $E_{\lambda_j} \subseteq \mathfrak{z}$ or $E_{\lambda_j} \subseteq \mathfrak{z}^\perp$, completing the proof.
\end{proof}

We now show that each ideal $\mathfrak{i}_j = E_{\lambda_j} \oplus \mathfrak{z}$ associated with a negative eigenvalue $\lambda_j$ inherits the quasicentroid condition when endowed with the restricted inner product.

\begin{theorem}\label{thm:ricci_restriction}
Let $(\mathfrak{n}, \langle \cdot, \cdot \rangle)$ be a nilpotent Riemannian Lie algebra with no abelian factor, such that $\operatorname{Ric} \in \operatorname{QC}(\mathfrak{n})$. Following the notation of Lemma \ref{lemma:ricci_eigenspace_splitting}, for each $1 \le j \le r$, consider the ideal $\mathfrak{i}_j = E_{\lambda_j} \oplus \mathfrak{z}$ endowed with the restricted inner product $\langle \cdot, \cdot \rangle_j = \langle \cdot, \cdot \rangle|_{\mathfrak{i}_j \times \mathfrak{i}_j}$. Then:
\begin{enumerate}
    \item[(i)] The center of $\mathfrak{i}_j$ is $\mathfrak{z}(\mathfrak{i}_j) = \mathfrak{z}$, and $\mathfrak{n}$ is the central product of the ideals $\mathfrak{i}_1, \dots, \mathfrak{i}_r$. In particular, $\mathfrak{n}$ is centrally decomposable if $r \ge 2$ (that is, if $\operatorname{Ric}$ has at least two distinct negative eigenvalues).
    \item[(ii)] The Ricci operator $\operatorname{Ric}_{\mathfrak{i}_j}$ of $(\mathfrak{i}_j, \langle \cdot, \cdot \rangle_j)$ satisfies
    \begin{equation*}
        \operatorname{Ric}_{\mathfrak{i}_j} = \lambda_j \operatorname{Id}_{\mathfrak{i}_j} + \varphi_j \in \mathbb{R}\operatorname{Id}_{\mathfrak{i}_j} \oplus L(\mathfrak{i}_j, \mathfrak{z}) \subseteq \operatorname{QC}(\mathfrak{i}_j).
    \end{equation*}
    In particular, the restricted metric on $\mathfrak{i}_j$ belongs to Gray's class $\mathcal{A}$. Its Ricci operator lies in $\operatorname{QC}(\mathfrak{i}_j)$ and has $\lambda_j$ as its unique negative eigenvalue on $\mathfrak{z}^\perp \cap \mathfrak{i}_j = E_{\lambda_j}$.
\end{enumerate}
\end{theorem}

\begin{proof}
To prove (i), clearly $\mathfrak{z} \subseteq \mathfrak{z}(\mathfrak{i}_j)$. Conversely, if $X + Z \in \mathfrak{z}(\mathfrak{i}_j)$ with $X \in E_{\lambda_j}$ and $Z \in \mathfrak{z}$, then $[X, E_{\lambda_j}] = \{0\}$. Since $[E_{\lambda_j}, E_{\lambda_k}] = \{0\}$ for all $k \neq j$, it follows that $[X, \mathfrak{n}] = \{0\}$, so $X \in \mathfrak{z} \cap E_{\lambda_j} = \{0\}$. Thus $\mathfrak{z}(\mathfrak{i}_j) = \mathfrak{z}$.

By Proposition \ref{prop:qc_diagonalizable_iff}, $\mathfrak{n}$ decomposes as the central product of the ideals $E_{\lambda_k} + \mathfrak{z}$ over all eigenvalues $1 \le k \le s$. However, for any positive eigenvalue ($k > r$), Lemma \ref{lemma:ricci_eigenspace_splitting} gives $E_{\lambda_k} \subseteq \mathfrak{z}$, meaning $E_{\lambda_k} + \mathfrak{z} = \mathfrak{z}$. These trivial factors do not contribute to the central product, so $\mathfrak{n}$ is already the central product of just $\mathfrak{i}_1, \dots, \mathfrak{i}_r$. Consequently, $\mathfrak{n}$ is centrally decomposable if there are at least two such ideals, meaning $r \ge 2$.

To prove (ii), fix $1 \le j \le r$. By Lemma \ref{lemma:ricci_eigenspace_splitting}, we have the orthogonal direct sum $\mathfrak{n} = \mathfrak{i}_j \oplus C_j$, where
\begin{equation*}
    C_j = \bigoplus_{\substack{k=1 \\ k \neq j}}^r E_{\lambda_k}.
\end{equation*}
By Proposition \ref{prop:qc_diagonalizable_iff}, $[E_{\lambda_k}, E_{\lambda_l}] = \{0\}$ for all $k \neq l$ and $[E_{\lambda_k}, E_{\lambda_k}] \subseteq E_{\lambda_k} \oplus \mathfrak{z}$ for every $k$. Since $\mathfrak{z}$ is the center of $\mathfrak{n}$, this implies that
\begin{equation*}
    [\mathfrak{i}_j, \mathfrak{i}_j] \subseteq \mathfrak{i}_j, \qquad [\mathfrak{i}_j, C_j] = \{0\}, \qquad \text{and} \qquad [C_j, C_j] \subseteq C_j \oplus \mathfrak{z}.
\end{equation*}
Let $\mathcal{B}_j$ and $\mathcal{B}_C$ be orthonormal bases of $\mathfrak{i}_j$ and $C_j$, respectively, so that the disjoint union $\mathcal{B} = \mathcal{B}_j \cup \mathcal{B}_C$ is an orthonormal basis of $\mathfrak{n}$. Since $\operatorname{Ric}$ is self-adjoint, polarizing \eqref{eq:ricci_formula} yields, for all $X, Y \in \mathfrak{n}$,
\begin{equation*}
    \langle \operatorname{Ric} X, Y \rangle = -\frac{1}{2} \sum_{u,v \in \mathcal{B}} \langle [X, u], v \rangle \langle [Y, u], v \rangle + \frac{1}{4} \sum_{u,v \in \mathcal{B}} \langle [u, v], X \rangle \langle [u, v], Y \rangle.
\end{equation*}
Now take $X \in E_{\lambda_j}$ and $Y \in \mathfrak{i}_j$. We examine the terms in the sums above when at least one of the basis vectors $u, v \in \mathcal{B}$ belongs to $\mathcal{B}_C$:
\begin{enumerate}
    \item[(a)] In the first sum, if $u \in \mathcal{B}_C$, then $[X, u] = 0$ because $X \in \mathfrak{i}_j$ and $[\mathfrak{i}_j, C_j] = \{0\}$. If $u \in \mathcal{B}_j$ and $v \in \mathcal{B}_C$, then $[X, u] \in [\mathfrak{i}_j, \mathfrak{i}_j] \subseteq \mathfrak{i}_j$, which is orthogonal to $v \in C_j$, so $\langle [X, u], v \rangle = 0$.
    \item[(b)] In the second sum, if one of $u, v$ lies in $\mathcal{B}_j$ and the other in $\mathcal{B}_C$, then $[u, v] = 0$ since $[\mathfrak{i}_j, C_j] = \{0\}$. If both $u, v \in \mathcal{B}_C$, then $[u, v] \in [C_j, C_j] \subseteq C_j \oplus \mathfrak{z}$. Since $X \in E_{\lambda_j}$ is orthogonal to both $C_j$ and $\mathfrak{z}$, we obtain $\langle [u, v], X \rangle = 0$.
\end{enumerate}
It follows from (a) and (b) that all terms with $u \in \mathcal{B}_C$ or $v \in \mathcal{B}_C$ vanish. Thus, for $X \in E_{\lambda_j}$ and $Y \in \mathfrak{i}_j$, the sums over $\mathcal{B}$ reduce to sums over $\mathcal{B}_j$:
\begin{equation*}
    \langle \operatorname{Ric} X, Y \rangle = -\frac{1}{2} \sum_{u,v \in \mathcal{B}_j} \langle [X, u], v \rangle \langle [Y, u], v \rangle + \frac{1}{4} \sum_{u,v \in \mathcal{B}_j} \langle [u, v], X \rangle \langle [u, v], Y \rangle.
\end{equation*}
Since $\mathfrak{i}_j$ is nilpotent and $\mathcal{B}_j$ is an orthonormal basis of $(\mathfrak{i}_j, \langle \cdot, \cdot \rangle_j)$, the right-hand side is precisely $\langle \operatorname{Ric}_{\mathfrak{i}_j} X, Y \rangle_j$. Moreover, since $X \in E_{\lambda_j}$, we have $\operatorname{Ric} X = \lambda_j X$, and therefore
\begin{equation*}
    \langle \operatorname{Ric}_{\mathfrak{i}_j} X, Y \rangle_j = \langle \operatorname{Ric} X, Y \rangle = \lambda_j \langle X, Y \rangle = \lambda_j \langle X, Y \rangle_j
\end{equation*}
for all $Y \in \mathfrak{i}_j$. This shows that $\operatorname{Ric}_{\mathfrak{i}_j} X = \lambda_j X$ for all $X \in E_{\lambda_j}$.

Finally, since $\operatorname{Ric}_{\mathfrak{i}_j}$ is self-adjoint with respect to $\langle \cdot, \cdot \rangle_j$ and preserves $E_{\lambda_j}$, it also preserves its orthogonal complement in $\mathfrak{i}_j$, namely $\mathfrak{z}$. Given any $V \in \mathfrak{i}_j$, write $V = X + Z$ with $X \in E_{\lambda_j}$ and $Z \in \mathfrak{z}$. Then
\begin{equation*}
    \operatorname{Ric}_{\mathfrak{i}_j}(V) = \operatorname{Ric}_{\mathfrak{i}_j}(X) + \operatorname{Ric}_{\mathfrak{i}_j}(Z) = \lambda_j X + \operatorname{Ric}_{\mathfrak{i}_j}(Z) = \lambda_j V + \varphi_j(V),
\end{equation*}
where $\varphi_j(V) = \operatorname{Ric}_{\mathfrak{i}_j}(Z) - \lambda_j Z \in \mathfrak{z}$. Since $\mathfrak{z}(\mathfrak{i}_j) = \mathfrak{z}$ by (i), we have $\varphi_j \in L(\mathfrak{i}_j, \mathfrak{z}(\mathfrak{i}_j) )$, and hence
\begin{equation*}
    \operatorname{Ric}_{\mathfrak{i}_j} = \lambda_j \operatorname{Id}_{\mathfrak{i}_j} + \varphi_j \in \mathbb{R}\operatorname{Id}_{\mathfrak{i}_j} \oplus L(\mathfrak{i}_j, \mathfrak{z}(\mathfrak{i}_j) ) \subseteq \operatorname{QC}(\mathfrak{i}_j).
\end{equation*}
This completes the proof.
\end{proof}

\begin{remark}\label{rem:central_de_rham}
Theorem \ref{thm:ricci_restriction} can be viewed as a central analogue of the de Rham decomposition in our setting. Because $\mathfrak{z} \subseteq [\mathfrak{n}, \mathfrak{n}]$, the Lie algebra $\mathfrak{n}$ cannot split as an orthogonal direct sum of non-trivial ideals unless it is already a direct product. Instead, the condition $\operatorname{Ric} \in \operatorname{QC}(\mathfrak{n})$ decomposes $\mathfrak{n}$ into a central product of ideals $\mathfrak{i}_j = E_{\lambda_j} \oplus \mathfrak{z}$ sharing the center $\mathfrak{z}$, such that the restricted metric on each piece preserves the property $\operatorname{Ric}_{\mathfrak{i}_j} \in \operatorname{QC}(\mathfrak{i}_j)$.

As discussed in the introduction, non-abelian nilpotent Lie groups do not admit left-invariant Einstein metrics. Consequently, a metric for which $\operatorname{Ric}$ has a single negative eigenvalue on $\mathfrak{z}^\perp$ (meaning $\operatorname{Ric}|_{\mathfrak{z}^\perp} = \lambda \operatorname{Id}_{\mathfrak{z}^\perp}$) may be regarded as one of the closest available analogs to an Einstein metric in this setting. Theorem \ref{thm:ricci_restriction} demonstrates how the assumption $\operatorname{Ric} \in \operatorname{QC}(\mathfrak{n})$ naturally reduces the general case to these fundamental building blocks.
\end{remark}

\begin{remark}\label{rem:ricci_center_sum}
It is worth emphasizing the relationship between the global Ricci operator $\operatorname{Ric}$ of $\mathfrak{n}$ and the restricted operators $\operatorname{Ric}_{\mathfrak{i}_j}$ of the ideals. For vectors in the non-central eigenspaces $E_{\lambda_j}$, the operators coincide, as $\operatorname{Ric} X = \operatorname{Ric}_{\mathfrak{i}_j} X = \lambda_j X$. However, this is not true on the shared center $\mathfrak{z}$. Because the bracket of two vectors in $\mathfrak{i}_k$ can have a non-zero component in $\mathfrak{z}$, each ideal contributes to the global curvature of the center. Evaluating \eqref{eq:ricci_formula} on $\mathfrak{z}$ reveals that the global Ricci operator on the center is precisely the sum of the restricted ones:
\begin{equation*}
    \operatorname{Ric}|_{\mathfrak{z}} = \sum_{k=1}^r \operatorname{Ric}_{\mathfrak{i}_k}|_{\mathfrak{z}}.
\end{equation*}
Consequently, the restricted metric on $\mathfrak{i}_j$ defines a genuinely new geometry on the center, as its Ricci operator loses positive curvature contributions given by the other ideals.
\end{remark}

As a direct consequence of Theorem \ref{thm:ricci_restriction}, the step of nilpotency of $\mathfrak{n}$ is controlled by the steps of nilpotency of the ideals $\mathfrak{i}_1, \dots, \mathfrak{i}_r$.

\begin{corollary}\label{cor:nilpotency_step}
Under the hypotheses and notation of Theorem \ref{thm:ricci_restriction}, for every integer $m \ge 1$ we have
\begin{equation*}
    \mathfrak{n}^m = \sum_{j=1}^r \mathfrak{i}_j^m.
\end{equation*}
In particular, the step of nilpotency of $\mathfrak{n}$ is equal to the maximum of the steps of nilpotency of the ideals $\mathfrak{i}_1, \dots, \mathfrak{i}_r$.
\end{corollary}

\begin{proof}
Since $\mathfrak{n} = \sum_{j=1}^r \mathfrak{i}_j$ and the cross-brackets vanish (i.e., $[\mathfrak{i}_j, \mathfrak{i}_k] = \{0\}$ for all $j \neq k$), the identity $\mathfrak{n}^m = \sum_{j=1}^r \mathfrak{i}_j^m$ follows immediately by induction on $m \ge 1$. Therefore, $\mathfrak{n}^m = \{0\}$ if and only if $\mathfrak{i}_j^m = \{0\}$ for all $1 \le j \le r$.
\end{proof}

\subsection{The Quasicentroid Condition when $\dim \mathfrak{z} = 1$}

We now apply Theorem \ref{thm:ricci_restriction} and Corollary \ref{cor:nilpotency_step} to classify nilpotent Riemannian Lie algebras with a one-dimensional center satisfying $\operatorname{Ric} \in \operatorname{QC}(\mathfrak{n})$. Recall that a two-step nilpotent Lie algebra with a one-dimensional center and no abelian factor is isomorphic to an odd-dimensional Heisenberg Lie algebra $\mathfrak{h}_{2m+1}$.

Given a Riemannian Lie algebra $(\mathfrak{n}, \langle \cdot, \cdot \rangle)$, for each $X \in \mathfrak{n}$ we define the skew-symmetric operator $J_X : \mathfrak{n} \to \mathfrak{n}$ by
\begin{equation}\label{eq:def_JX}
    \langle J_X Y, Z \rangle = \langle X, [Y, Z] \rangle \quad \text{for all } Y, Z \in \mathfrak{n}.
\end{equation}
Although in the literature these $J$-operators are typically used in the specific context of two-step nilpotent Lie algebras, where $X$ is taken exclusively from the center $\mathfrak{z}$ and the operators act only on $\mathfrak{z}^\perp$, it is much more convenient for our general analysis to define them globally on the entire algebra $\mathfrak{n}$.

Under this extended definition, note that $J_X = 0$ whenever $X \in [\mathfrak{n}, \mathfrak{n}]^\perp$. Moreover, for every $X \in \mathfrak{n}$, the center $\mathfrak{z}$ is contained in $\ker J_X$. Because $J_X$ is skew-symmetric, this kernel condition immediately forces the image to satisfy $\operatorname{Im} J_X \subseteq \mathfrak{z}^\perp$.

For self-adjoint operators, membership in the quasicentroid admits a natural characterization in terms of the operators $J_X$.

\begin{proposition}\label{lemma:qc_commutes_JX}
Let $(\mathfrak{n}, \langle \cdot, \cdot \rangle)$ be a Riemannian Lie algebra and let $T : \mathfrak{n} \to \mathfrak{n}$ be a self-adjoint linear operator. Then $T \in \operatorname{QC}(\mathfrak{n})$ if and only if $[T, J_X] = 0$ for all $X \in \mathfrak{n}$.
\end{proposition}

\begin{proof}
Since $T$ is self-adjoint, for all $X, Y, Z \in \mathfrak{n}$ we have
\begin{equation*}
    \langle J_X(TY), Z \rangle = \langle X, [TY, Z] \rangle \quad \text{and} \quad \langle T(J_X Y), Z \rangle = \langle J_X Y, TZ \rangle = \langle X, [Y, TZ] \rangle.
\end{equation*}
Consequently, $\langle (J_X T - T J_X)Y, Z \rangle = \langle X, [TY, Z] - [Y, TZ] \rangle$ for all $X, Y, Z \in \mathfrak{n}$. By non-degeneracy of the inner product, $J_X T = T J_X$ for all $X \in \mathfrak{n}$ if and only if $[TY, Z] = [Y, TZ]$ for all $Y, Z \in \mathfrak{n}$, that is, $T \in \operatorname{QC}(\mathfrak{n})$.
\end{proof}

The proof of the following theorem proceeds as follows. We first apply Theorem \ref{thm:ricci_restriction} to reduce the problem to the case where $\operatorname{Ric}$ has exactly one negative eigenvalue on $\mathfrak{z}^\perp$. Assuming by contradiction that the step of nilpotency is $k \ge 3$, we select a unit vector $u \in [\mathfrak{n}, \mathfrak{n}] \cap \mathfrak{z}^\perp$ belonging to the \textit{second center} of $\mathfrak{n}$, and compare its Ricci curvature with that of $v = J_z u$, where $z$ is a unit vector spanning $\mathfrak{z}$. Using the Jacobi identity, we show that $v$ is non-zero and orthogonal to $[\mathfrak{n}, \mathfrak{n}]$ (and hence also to $\mathfrak{z}$). This orthogonality significantly simplifies the Ricci formula in the direction of $v$, ultimately forcing $J_u = 0$, which contradicts the assumption that $u \in [\mathfrak{n}, \mathfrak{n}]$.

\begin{theorem}\label{thm:center_one_heisenberg}
Let $(\mathfrak{n}, \langle \cdot, \cdot \rangle)$ be a nilpotent Riemannian Lie algebra with no abelian factor and $\dim \mathfrak{z} = 1$. If $\operatorname{Ric} \in \operatorname{QC}(\mathfrak{n})$, then $\mathfrak{n}$ is two-step nilpotent, and therefore $\mathfrak{n}$ is isomorphic to an odd-dimensional Heisenberg Lie algebra.
\end{theorem}

\begin{proof}
By Theorem \ref{thm:ricci_restriction}, for each negative eigenvalue $\lambda_j$ ($1 \le j \le r$), the ideal $(\mathfrak{i}_j, \langle \cdot, \cdot \rangle_j)$ is a nilpotent Riemannian Lie algebra with one-dimensional center $\mathfrak{z}(\mathfrak{i}_j) = \mathfrak{z}$ (hence with no abelian factor), satisfying $\operatorname{Ric}_{\mathfrak{i}_j} \in \operatorname{QC}(\mathfrak{i}_j)$ and having $\lambda_j < 0$ as its unique negative eigenvalue on $\mathfrak{z}^\perp \cap \mathfrak{i}_j = E_{\lambda_j}$. By Corollary \ref{cor:nilpotency_step}, to prove that $\mathfrak{n}$ is two-step nilpotent, it suffices to show that each ideal $\mathfrak{i}_j$ is two-step nilpotent. We may therefore assume without loss of generality that $(\mathfrak{n}, \langle \cdot, \cdot \rangle)$ has $\dim \mathfrak{z} = 1$ and that there exists $\lambda < 0$ such that
\begin{equation}\label{eq:ricci_lambda_zperp}
    \operatorname{Ric} X = \lambda X \quad \text{for all } X \in \mathfrak{z}^\perp.
\end{equation}

Notice first that for any orthonormal basis $\{e_1, \dots, e_n\}$ of $\mathfrak{n}$ and any $X \in \mathfrak{n}$, the Hilbert-Schmidt norms of the operators $\operatorname{ad}_X$ and $J_X$ are given by
\begin{align*}
    \|\operatorname{ad}_X\|^2 &= \sum_{p=1}^n \|[X, e_p]\|^2 = \sum_{p,q=1}^n \langle [X, e_p], e_q \rangle^2, \\
    \|J_X\|^2 &= \sum_{p,q=1}^n \langle J_X e_p, e_q \rangle^2 = \sum_{p,q=1}^n \langle X, [e_p, e_q] \rangle^2.
\end{align*}
Thus, the Ricci formula \eqref{eq:ricci_formula} can be written in coordinate-free form as
\begin{equation}\label{eq:ricci_JX}
    \langle \operatorname{Ric} X, X \rangle = \frac{1}{4} \|J_X\|^2 - \frac{1}{2} \|\operatorname{ad}_X\|^2.
\end{equation}

Suppose, for a contradiction, that the step of nilpotency of $\mathfrak{n}$ is $k \ge 3$, so that $\mathfrak{n}^{k-1} \neq \{0\}$ and $\mathfrak{n}^k = \{0\}$. Because $k \ge 3$, it follows that $\mathfrak{n}^{k-2} \subseteq \mathfrak{n}^1 = [\mathfrak{n}, \mathfrak{n}]$. Additionally, $[\mathfrak{n}^{k-2}, \mathfrak{n}] = \mathfrak{n}^{k-1} \subseteq \mathfrak{z}$, which means that $\mathfrak{n}^{k-2}$ is contained in $\mathfrak{z}_{2}$, the second center of $\mathfrak{n}$. Moreover, $\mathfrak{n}^{k-2} \not\subseteq \mathfrak{z}$ since $\mathfrak{n}^{k-1} \neq \{0\}$. By choosing a non-zero vector in $\mathfrak{n}^{k-2}$ orthogonal to $\mathfrak{z}$ and normalizing it, we obtain a unit vector $u \in [\mathfrak{n}, \mathfrak{n}] \cap \mathfrak{z}^\perp$ satisfying $[u, \mathfrak{n}] \subseteq \mathfrak{z}$. Since $u \in \mathfrak{z}^\perp$ and $\|u\| = 1$, equation \eqref{eq:ricci_lambda_zperp} yields $\operatorname{Ric} u = \lambda u$, and consequently $\langle \operatorname{Ric} u, u \rangle = \lambda$.

Let $z \in \mathfrak{z}$ be a unit vector spanning $\mathfrak{z}$, and define $v = J_z u$. For any $Y \in \mathfrak{n}$, since $[u, Y] \in \mathbb{R} z$, definition \eqref{eq:def_JX} yields
\begin{equation}\label{eq:bracket_u_Y}
    [u, Y] = \langle z, [u, Y] \rangle z = \langle J_z u, Y \rangle z = \langle v, Y \rangle z.
\end{equation}
Since $u \notin \mathfrak{z}$, the map $\operatorname{ad}_u$ is non-zero, so \eqref{eq:bracket_u_Y} implies $v \neq 0$. Furthermore, for all $Y, W \in \mathfrak{n}$, both $[u, Y]$ and $[u, W]$ lie in $\mathfrak{z}$, so the Jacobi identity gives
\begin{equation*}
    [u, [Y, W]] = [[u, Y], W] + [Y, [u, W]] = 0.
\end{equation*}
Hence $[u, [\mathfrak{n}, \mathfrak{n}]] = \{0\}$, and for every $W \in [\mathfrak{n}, \mathfrak{n}]$ we have $\langle v, W \rangle = \langle J_z u, W \rangle = \langle z, [u, W] \rangle = 0$. This shows that $v \in [\mathfrak{n}, \mathfrak{n}]^\perp \subseteq \mathfrak{z}^\perp$. In particular, $J_v = 0$, and by \eqref{eq:ricci_lambda_zperp} we also have $\operatorname{Ric} v = \lambda v$.

Now choose an orthonormal basis $\{e_1, \dots, e_n\}$ of $\mathfrak{n}$ with $e_1 = u$. Evaluating \eqref{eq:ricci_JX} at $X = v$, using that $J_v = 0$, and keeping only the term $p = 1$ ($e_1 = u$) in the expansion of $\|\operatorname{ad}_v\|^2$, we obtain
\begin{equation*}
    \lambda = \frac{\langle \operatorname{Ric} v, v \rangle}{\|v\|^2} = -\frac{1}{2\|v\|^2}\|\operatorname{ad}_v\|^2 = -\frac{1}{2\|v\|^2} \sum_{p=1}^n \|[v, e_p]\|^2 \le -\frac{1}{2} \frac{\|[v, u]\|^2}{\|v\|^2}.
\end{equation*}
By \eqref{eq:bracket_u_Y}, we have $[u, v] = \|v\|^2 z$, so $\|[v, u]\|^2 = \|v\|^4$, and the inequality above simplifies to
\begin{equation}\label{eq:bound_v}
    \lambda \le -\frac{1}{2} \|v\|^2.
\end{equation}
On the other hand, by \eqref{eq:bracket_u_Y} we have $\|[u, e_p]\|^2 = \langle v, e_p \rangle^2$ for each $1 \le p \le n$, so $\|\operatorname{ad}_u\|^2 = \sum_{p=1}^n \|[u, e_p]\|^2 = \|v\|^2$. Evaluating \eqref{eq:ricci_JX} at the unit vector $X = u$ we have
\begin{equation}\label{eq:ricci_u}
    \lambda = \langle \operatorname{Ric} u, u \rangle = \frac{1}{4} \|J_u\|^2 - \frac{1}{2} \|v\|^2.
\end{equation}

Combining \eqref{eq:ricci_u} and \eqref{eq:bound_v}, we obtain
\begin{equation*}
    \frac{1}{4} \|J_u\|^2 - \frac{1}{2} \|v\|^2 \le -\frac{1}{2} \|v\|^2,
\end{equation*}
which forces $\frac{1}{4} \|J_u\|^2 \le 0$, and hence $J_u = 0$. By \eqref{eq:def_JX}, $J_u = 0$ implies that $u \in [\mathfrak{n}, \mathfrak{n}]^\perp$, contradicting that $u$ is a unit vector in $[\mathfrak{n}, \mathfrak{n}]$.

Consequently, $\mathfrak{n}$ must be two-step nilpotent. Since $\dim \mathfrak{z} = 1$ and $\mathfrak{n}$ has no abelian factor, it follows that $\mathfrak{n}$ is isomorphic to an odd-dimensional Heisenberg Lie algebra.
\end{proof}

\begin{remark}\label{rem:heisenberg_gray_A}
It is straightforward to verify that any inner product on the $(2m+1)$-dimensional Heisenberg Lie algebra $\mathfrak{h}_{2m+1}$ defines a left-invariant metric belonging to Gray's class $\mathcal{A}$.
\end{remark}

As established earlier, every filiform Lie algebra has a one-dimensional center and no abelian factor. Consequently, Theorem \ref{thm:center_one_heisenberg} immediately rules out the existence of Riemannian metrics with $\operatorname{Ric} \in \operatorname{QC}(\mathfrak{n})$ on any non-Heisenberg filiform Lie algebra.

\begin{corollary}\label{cor:filiform}
Let $\mathfrak{n}$ be a filiform Lie algebra of dimension $n \ge 4$. Then there exists no inner product $\langle \cdot, \cdot \rangle$ on $\mathfrak{n}$ such that the associated Riemannian Lie algebra $(\mathfrak{n}, \langle \cdot, \cdot \rangle)$ satisfies $\operatorname{Ric} \in \operatorname{QC}(\mathfrak{n})$.
\end{corollary}

We conclude this section by showing that the hypothesis $\dim \mathfrak{z} = 1$ in Theorem \ref{thm:center_one_heisenberg} is sharp. Indeed, as soon as $\dim \mathfrak{z} = 2$, a nilpotent Riemannian Lie algebra can satisfy $\operatorname{Ric} \in \operatorname{QC}(\mathfrak{n})$ (and even have a unique negative eigenvalue on $\mathfrak{z}^\perp$) while having nilpotency step greater than two. We illustrate this phenomenon with two explicit examples of nilpotency steps three and five.

\begin{example}\label{ex:free_3step_dim5}
Let $\mathfrak{n}$ be the $5$-dimensional $3$-step free nilpotent Lie algebra on two generators. We endow it with the inner product for which $\{e_1, \dots, e_5\}$ is an orthonormal basis and the non-zero Lie brackets are given by
\begin{equation*}
    [e_1, e_2] = \frac{\sqrt{6}}{3}e_3, \qquad [e_1, e_3] = \frac{2\sqrt{3}}{3}e_4, \qquad [e_2, e_3] = \frac{2\sqrt{3}}{3}e_5.
\end{equation*}
Here $\mathfrak{z} = \operatorname{span}\{e_4, e_5\}$ has dimension two and $\mathfrak{z}^\perp = \operatorname{span}\{e_1, e_2, e_3\}$. The basis $\{e_1, \dots, e_5\}$ is \emph{nice} in the sense of Definition \ref{def:nice_basis} below. Consequently, the operators $\operatorname{ad}_{e_i}$ and $\operatorname{ad}_{e_j}$ are mutually orthogonal for $i \neq j$ with respect to the Hilbert-Schmidt inner product, as are $J_{e_i}$ and $J_{e_j}$. Thus, \eqref{eq:ricci_JX} implies that $\operatorname{Ric}$ is diagonal with respect to this basis. Evaluating \eqref{eq:ricci_JX} on each basis vector gives
\begin{align*}
    \langle \operatorname{Ric} e_1, e_1 \rangle &= \langle \operatorname{Ric} e_2, e_2 \rangle = -\frac{1}{2}\left(\frac{2}{3} + \frac{4}{3}\right) = -1, \\
    \langle \operatorname{Ric} e_3, e_3 \rangle &= \frac{1}{2}\left(\frac{2}{3}\right) - \frac{1}{2}\left(\frac{4}{3} + \frac{4}{3}\right) = -1, \\
    \langle \operatorname{Ric} e_4, e_4 \rangle &= \langle \operatorname{Ric} e_5, e_5 \rangle = \frac{1}{2}\left(\frac{4}{3}\right) = \frac{2}{3}.
\end{align*}
Hence $\operatorname{Ric} = \operatorname{diag}\bigl(-1, -1, -1, \frac{2}{3}, \frac{2}{3}\bigr)$, which satisfies $\operatorname{Ric}|_{\mathfrak{z}^\perp} = -\operatorname{Id}_{\mathfrak{z}^\perp}$. In particular, $\operatorname{Ric} \in \mathbb{R}\operatorname{Id}_{\mathfrak{n}} \oplus L(\mathfrak{n}, \mathfrak{z}) \subseteq \operatorname{QC}(\mathfrak{n})$.
This specific example corresponds to a particular $\mathcal{A}$-metric from the family described in \cite[Proposition 3.10]{ChenChenZhu2026}.
\end{example}

\begin{example}\label{ex:5step_dim7}
Consider the $7$-dimensional nilpotent Lie algebra denoted by $(23457C)$ in Gong's classification \cite{Gong1998}. We endow it with the inner product for which $\{e_1, \dots, e_7\}$ is an orthonormal basis and the non-zero Lie brackets are defined by
\begin{align*}
    [e_1, e_2] &= \frac{2\sqrt{11}}{11}\,e_3, & [e_1, e_3] &= \frac{\sqrt{10}}{5}\,e_4, & [e_1, e_4] &= \frac{2\sqrt{330}}{55}\,e_5, \\
    [e_1, e_5] &= \frac{2\sqrt{5}}{5}\,e_6, & [e_2, e_5] &= \frac{3\sqrt{22}}{11}\,e_7, & [e_3, e_4] &= -\frac{6\sqrt{165}}{55}\,e_7.
\end{align*}
The Lie algebra $\mathfrak{n}$ is a $5$-step nilpotent Lie algebra generated by $e_1$ and $e_2$ with two-dimensional center $\mathfrak{z} =  \operatorname{span}\{e_6, e_7\}$ and $\mathfrak{z}^\perp = \operatorname{span}\{e_1, \dots, e_5\}$.

As in Example \ref{ex:free_3step_dim5}, the basis $\{e_1, \dots, e_7\}$ is nice, thus $\operatorname{Ric}$ is diagonal in this basis. Evaluating \eqref{eq:ricci_JX} on the basis of $\mathfrak{z}^\perp$, we obtain
\begin{align*}
    \langle \operatorname{Ric} e_1, e_1 \rangle &= -\frac{1}{2}\left(\frac{4}{11} + \frac{2}{5} + \frac{24}{55} + \frac{4}{5}\right) = -1, \\
    \langle \operatorname{Ric} e_2, e_2 \rangle &= -\frac{1}{2}\left(\frac{4}{11} + \frac{18}{11}\right) = -1, \\
    \langle \operatorname{Ric} e_3, e_3 \rangle &= \frac{1}{2}\left(\frac{4}{11}\right) - \frac{1}{2}\left(\frac{2}{5} + \frac{108}{55}\right) = -1, \\
    \langle \operatorname{Ric} e_4, e_4 \rangle &= \frac{1}{2}\left(\frac{2}{5}\right) - \frac{1}{2}\left(\frac{24}{55} + \frac{108}{55}\right) = -1, \\
    \langle \operatorname{Ric} e_5, e_5 \rangle &= \frac{1}{2}\left(\frac{24}{55}\right) - \frac{1}{2}\left(\frac{4}{5} + \frac{18}{11}\right) = -1,
\end{align*}
while on the center $\mathfrak{z} = \operatorname{span}\{e_6, e_7\}$ we have
\begin{equation*}
    \langle \operatorname{Ric} e_6, e_6 \rangle = \frac{1}{2}\left(\frac{4}{5}\right) = \frac{2}{5} \quad \text{and} \quad \langle \operatorname{Ric} e_7, e_7 \rangle = \frac{1}{2}\left(\frac{18}{11} + \frac{108}{55}\right) = \frac{9}{5}.
\end{equation*}
Therefore,
\begin{equation*}
    \operatorname{Ric} = \operatorname{diag}\left(-1, -1, -1, -1, -1, \frac{2}{5}, \frac{9}{5}\right) \in \mathbb{R}\operatorname{Id}_{\mathfrak{n}} \oplus L(\mathfrak{n}, \mathfrak{z}) \subseteq \operatorname{QC}(\mathfrak{n}),
\end{equation*}
providing a $5$-step nilpotent example with two-dimensional center and $\operatorname{Ric}|_{\mathfrak{z}^\perp} = -\operatorname{Id}_{\mathfrak{z}^\perp}$.
\end{example}

\section{From Gray $\mathcal{A}$-metrics to the Quasicentroid Condition}

As established earlier, the condition $\operatorname{Ric} \in \operatorname{QC}(\mathfrak{n})$ is formally stronger than being a Gray $\mathcal{A}$-metric. In this section, we show that these two geometric conditions are in fact equivalent, provided we impose some natural structural or combinatorial assumptions on the algebra.

Let $(\mathfrak{n}, \langle \cdot, \cdot \rangle)$ be a Riemannian Lie algebra and let $\{e_1, \dots, e_n\}$ be an orthonormal basis which diagonalizes its Ricci operator; we have $\operatorname{Ric} e_i = \lambda_i e_i$ for $1 \le i \le n$.

If the metric is a Gray $\mathcal{A}$-metric, polarizing the cubic condition $\langle [\operatorname{Ric} X, X], X \rangle = 0$ yields the fundamental cyclic identity
\begin{equation}\label{eq:gray_cyclic}
    (\lambda_i - \lambda_j) \langle [e_i, e_j], e_k \rangle + (\lambda_j - \lambda_k) \langle [e_j, e_k], e_i \rangle + (\lambda_k - \lambda_i) \langle [e_k, e_i], e_j \rangle = 0,
\end{equation}
for all $1 \le i, j, k \le n$. This identity serves as the cornerstone for our subsequent results.

\subsection{Equivalence in the Two-Step Nilpotent Case}

A fundamental property of any two-step nilpotent Riemannian Lie algebra $(\mathfrak{n},\langle \cdot, \cdot\rangle)$ is that the Ricci operator automatically preserves the center of $\mathfrak{n}$, say $\mathfrak{z}$ (and consequently its orthogonal complement $\mathfrak{z}^\perp$); this follows directly from evaluating the Ricci operator formula, which trivially yields $\langle \operatorname{Ric} u, z \rangle = 0$ for any $u \in \mathfrak{z}^\perp$ and $z \in \mathfrak{z}$. This automatic block-diagonal structure of $\operatorname{Ric}$ was utilized in \cite[Theorem 2.2]{ChenChenZhangZhu2025} and \cite[Theorem 4]{AriasMarcoFB2026} to characterize Gray $\mathcal{A}$-metrics by the condition that the restriction of the Ricci operator to $\mathfrak{z}^\perp$ commutes with the endomorphisms $J_z|_{\mathfrak{z}^\perp}$ for all $z \in \mathfrak{z}$.

For completeness and to place this characterization within our framework, we reproduce the argument given in these works, adapting the notation to our setting. This reformulation makes explicit that the restricted commutativity condition they observed is, in fact, precisely the requirement that the Ricci operator belongs to the quasicentroid of the Lie algebra.

\begin{theorem}\label{thm:grayA_2step}
Let $(\mathfrak{n}, \langle \cdot, \cdot \rangle)$ be a two-step nilpotent Riemannian Lie algebra. Then the metric belongs to Gray's class $\mathcal{A}$ if and only if its Ricci operator satisfies $\operatorname{Ric} \in \operatorname{QC}(\mathfrak{n})$.
\end{theorem}

\begin{proof}
Let $X \in \mathfrak{n}$ and let $\mathfrak{z}$ be the center of $\mathfrak{n}$. We decompose $X = u + z$ with $u \in \mathfrak{z}^\perp$ and $z \in \mathfrak{z}$. Since $\mathfrak{n}$ is two-step nilpotent, $[\mathfrak{n}, \mathfrak{n}] \subseteq \mathfrak{z}$. Hence, for any $Y, Z \in \mathfrak{n}$, $\langle J_u Y, Z \rangle = \langle u, [Y, Z] \rangle = 0$ because $[Y, Z] \in \mathfrak{z}$ and $u \in \mathfrak{z}^\perp$. This implies that $J_u = 0$ for all $u \in \mathfrak{z}^\perp$, and thus $J_X = J_u + J_z = J_z$.

As noted above, $\operatorname{Ric}$ preserves the orthogonal splitting $\mathfrak{n} = \mathfrak{z}^\perp \oplus \mathfrak{z}$. Now we consider the Gray $\mathcal{A}$-metric condition \eqref{grayA_condition}. Since $\mathfrak{n}$ is two-step nilpotent, $[\operatorname{Ric} X, X] \in \mathfrak{z}$. Taking the inner product with $X = u + z$ yields
\begin{equation*}
    \langle [\operatorname{Ric} X, X], X \rangle = \langle [\operatorname{Ric} X, X], u + z \rangle = \langle [\operatorname{Ric} X, X], z \rangle = \langle J_z(\operatorname{Ric} X), X \rangle.
\end{equation*}
Thus, the condition $\langle [\operatorname{Ric} X, X], X \rangle = 0$ for all $X \in \mathfrak{n}$ is equivalent to $\langle J_z \operatorname{Ric} X, X \rangle = 0$ for all $X \in \mathfrak{n}$ and $z \in \mathfrak{z}$.

Since $\operatorname{Ric}$ is self-adjoint and $J_z$ is skew-symmetric, the quadratic form $X \mapsto \langle J_z \operatorname{Ric} X, X \rangle$ vanishes identically if and only if $J_z \operatorname{Ric}$ is a skew-symmetric operator, which is equivalent to $[J_z, \operatorname{Ric}] = 0$.

Therefore, the $\mathcal{A}$-metric condition holds if and only if $[J_z, \operatorname{Ric}] = 0$ for all $z \in \mathfrak{z}$. Finally, since $J_u = 0$ for all $u \in \mathfrak{z}^\perp$, it follows that $[J_X, \operatorname{Ric}] = 0$ for all $X \in \mathfrak{n}$, which by Proposition \ref{lemma:qc_commutes_JX} is exactly the statement that $\operatorname{Ric} \in \operatorname{QC}(\mathfrak{n})$.
\end{proof}

\subsection{Nilpotent Lie algebras admitting a nice basis}

For nilpotency classes greater than two, condition \eqref{grayA_condition} still implies that $\operatorname{Ric} \in \operatorname{QC}(\mathfrak{n})$, provided that $(\mathfrak{n},\langle\cdot,\cdot\rangle)$ admits a \textit{nice basis}. Although the term \emph{nice basis} was formally coined by Yuri Nikolayevsky \cite[Definition 3]{Nikolayevsky2011} in his study of Einstein nilradicals, related ideas can be traced back, to the best of our knowledge, to the \textit{uniform bases} introduced by Edward Deloff \cite{Deloff1979} and subsequently studied by Thomas Wolter \cite{Wolter1991}.

\begin{definition}\label{def:nice_basis}
An orthonormal basis $\{e_1, \dots, e_n\}$ of a Riemannian nilpotent Lie algebra $(\mathfrak{n}, \langle \cdot, \cdot \rangle)$ is called a \emph{nice basis} if it satisfies the following two conditions:
\begin{enumerate}
    \item[(1)] For all $i, j$, there is at most one index $k$ such that $\langle [e_i, e_j], e_k \rangle \neq 0$.
    \item[(2)] For all $i, k$, there is at most one index $j$ such that $\langle [e_i, e_j], e_k \rangle \neq 0$.
\end{enumerate}
\end{definition}

Admitting a nice basis imposes a combinatorial rigidity on $(\mathfrak{n}, \langle \cdot, \cdot \rangle)$ much stronger than the mere Hilbert-Schmidt orthogonality of its structural operators, $\operatorname{ad}_{e_i}$ and $J_{e_i}$. Specifically, for any $i \neq j$, the matrix representations of the adjoint operators $\operatorname{ad}_{e_i}$ and $\operatorname{ad}_{e_j}$ in this basis have disjoint \emph{supports}; that is, they never share a non-zero entry in the same position. The exact same non-overlapping property holds for the endomorphisms $J_{e_i}$ and $J_{e_j}$. This structural disjointness trivially forces all their mutual Hilbert-Schmidt inner products to vanish, which in turn guarantees that the Ricci operator is diagonal in this basis.

\begin{theorem}\label{thm:grayA_nice_basis}
Let $(\mathfrak{n}, \langle \cdot, \cdot \rangle)$ be a nilpotent Riemannian Lie algebra admitting a nice basis. If the metric belongs to Gray's class $\mathcal{A}$, then its Ricci operator satisfies $\operatorname{Ric} \in \operatorname{QC}(\mathfrak{n})$.
\end{theorem}

\begin{proof}
Let $\{e_1, \dots, e_n\}$ be a nice basis of $(\mathfrak{n}, \langle \cdot, \cdot \rangle)$, and so $\operatorname{Ric}$ is diagonal with respect to this basis and we have $\operatorname{Ric} e_i = \lambda_i e_i$. Therefore, the cyclic identity \eqref{eq:gray_cyclic} holds.

Suppose that $[e_i, e_j] \neq 0$. By condition (1) of Definition \ref{def:nice_basis}, there exists a unique $k$ such that $\langle [e_i, e_j], e_k \rangle \neq 0$. We claim that $\langle [e_j, e_k], e_i \rangle = 0$ and $\langle [e_k, e_i], e_j \rangle = 0$. If, for instance, $\langle [e_j, e_k], e_i \rangle \neq 0$, then $[e_j, e_k]$ is proportional to $e_i$. But this implies that $\operatorname{ad}_{e_j}^2(e_i)$ is a non-zero multiple of $e_i$, which contradicts the fact that the operator $\operatorname{ad}_{e_j}$ is nilpotent.

Therefore, the second and third terms in the polarized Gray identity given in \eqref{eq:gray_cyclic} vanish identically, reducing the equation to
\begin{equation*}
    (\lambda_i - \lambda_j) \langle [e_i, e_j], e_k \rangle = 0.
\end{equation*}
Since $\langle [e_i, e_j], e_k \rangle \neq 0$, we conclude that $\lambda_i = \lambda_j$. By bilinearity, this implies that $[E_{\lambda}, E_{\mu}] = \{0\}$ for any two distinct eigenspaces of $\operatorname{Ric}$, which by Proposition \ref{prop:qc_diagonalizable_iff} is equivalent to $\operatorname{Ric} \in \operatorname{QC}(\mathfrak{n})$.
\end{proof}

\subsection{A trace characterization}

For an arbitrary nilpotent Riemannian Lie algebra $(\mathfrak{n}, \langle \cdot, \cdot \rangle)$, we can establish a general trace characterization for metrics in the Gray class $\mathcal{A}$. By combining the polarized Gray identity given in \eqref{eq:gray_cyclic} with the vanishing of the Killing form, we obtain a necessary and sufficient condition for the Ricci operator to satisfy $\operatorname{Ric} \in \operatorname{QC}(\mathfrak{n})$.

\begin{theorem}\label{thm:grayA_trace}
Let $(\mathfrak{n}, \langle \cdot, \cdot \rangle)$ be a nilpotent Riemannian Lie algebra. Then $\operatorname{Ric} \in \operatorname{QC}(\mathfrak{n})$ if and only if the metric is a Gray $\mathcal{A}$-metric and satisfies
\begin{equation}\label{eq:trace_conditions}
    \operatorname{tr}(\operatorname{ad}_X \operatorname{ad}_Y \operatorname{Ric}) = 0 \quad \text{and} \quad \operatorname{tr}(\operatorname{ad}_X \operatorname{Ric} \operatorname{ad}_Y \operatorname{Ric}) = 0
\end{equation}
for all $X, Y \in \mathfrak{n}$.
\end{theorem}

\begin{proof}
Suppose first that $\operatorname{Ric} \in \operatorname{QC}(\mathfrak{n})$. Then $[\operatorname{Ric} X, X] = 0$ for all $X \in \mathfrak{n}$, so the metric is a Gray $\mathcal{A}$-metric. Moreover, the quasicentroid condition $[\operatorname{Ric} Y, Z] = [Y, \operatorname{Ric} Z]$ for all $Y, Z \in \mathfrak{n}$ is equivalent to the operator identity $\operatorname{ad}_Y  \operatorname{Ric} = \operatorname{ad}_{\operatorname{Ric} Y}$. Since $\mathfrak{n}$ is nilpotent, its Killing form vanishes identically, that is, $\operatorname{tr}(\operatorname{ad}_U  \operatorname{ad}_V) = 0$ for all $U, V \in \mathfrak{n}$. Consequently, for any $X, Y \in \mathfrak{n}$,
\begin{align*}
    \operatorname{tr}(\operatorname{ad}_X  \operatorname{ad}_Y  \operatorname{Ric}) &= \operatorname{tr}(\operatorname{ad}_X  \operatorname{ad}_{\operatorname{Ric} Y}) = 0, \\
    \operatorname{tr}(\operatorname{ad}_X  \operatorname{Ric}  \operatorname{ad}_Y  \operatorname{Ric}) &= \operatorname{tr}(\operatorname{ad}_{\operatorname{Ric} X}  \operatorname{ad}_{\operatorname{Ric} Y}) = 0.
\end{align*}

Conversely, assume that the metric is a Gray $\mathcal{A}$-metric and that \eqref{eq:trace_conditions} holds. Let $\{e_1, \dots, e_n\}$ be an orthonormal eigenbasis for $\operatorname{Ric}$ with corresponding eigenvalues $\lambda_1, \dots, \lambda_n$, meaning $\operatorname{Ric} e_i = \lambda_i e_i$. Set the structure constants $C_{i,j,k} = \langle [e_i, e_j], e_k \rangle$, noting the skew-symmetry $C_{i,j,k} = -C_{j,i,k}$.

By the cyclic identity \eqref{eq:gray_cyclic}, we have
\begin{equation*}
    (\lambda_i - \lambda_j) C_{i,j,k} + (\lambda_j - \lambda_k) C_{j,k,i} = -(\lambda_k - \lambda_i) C_{k,i,j},
\end{equation*}
for all $i, j, k$. Squaring both sides and summing over all $1 \le i, j, k \le n$, we obtain
\begin{equation*}
    \sum_{i,j,k} (\lambda_i - \lambda_j)^2 C_{i,j,k}^2 + \sum_{i,j,k} (\lambda_j - \lambda_k)^2 C_{j,k,i}^2 + 2 \sum_{i,j,k} (\lambda_i - \lambda_j)(\lambda_j - \lambda_k) C_{i,j,k} C_{j,k,i} = \sum_{i,j,k} (\lambda_k - \lambda_i)^2 C_{k,i,j}^2.
\end{equation*}
By relabeling the summation indices, the second sum on the left-hand side cancels with the right-hand side, leaving
\begin{equation}\label{eq:squared_cyclic_sum}
    \sum_{i,j,k=1}^n (\lambda_i - \lambda_j)^2 C_{i,j,k}^2 + 2 \sum_{i,j,k=1}^n (\lambda_i - \lambda_j)(\lambda_j - \lambda_k) C_{i,j,k} C_{j,k,i} = 0.
\end{equation}

To express the second sum in \eqref{eq:squared_cyclic_sum} in terms of traces, fix $j \in \{1, \dots, n\}$ and consider the operator $T_j = \operatorname{ad}_{e_j} (\lambda_j \operatorname{Id}_{\mathfrak{n}} - \operatorname{Ric})$. For any basis vector $e_i$, we have
\begin{equation*}
    T_j(e_i) = (\lambda_j - \lambda_i)[e_j, e_i] = (\lambda_i - \lambda_j)[e_i, e_j] = \sum_{k=1}^n (\lambda_i - \lambda_j) C_{i,j,k} e_k.
\end{equation*}
Thus, the matrix entries of $T_j$ with respect to $\{e_1, \dots, e_n\}$ are $\langle T_j e_i, e_k \rangle = (\lambda_i - \lambda_j) C_{i,j,k}$ and $\langle T_j e_k, e_i \rangle = (\lambda_k - \lambda_j) C_{k,j,i} = (\lambda_j - \lambda_k) C_{j,k,i}$. Consequently,
\begin{equation*}
    \operatorname{tr}(T_j^2) = \sum_{i,k=1}^n \langle T_j e_i, e_k \rangle \langle T_j e_k, e_i \rangle = \sum_{i,k=1}^n (\lambda_i - \lambda_j)(\lambda_j - \lambda_k) C_{i,j,k} C_{j,k,i}.
\end{equation*}
On the other hand, expanding $T_j^2 = \operatorname{ad}_{e_j}  (\lambda_j \operatorname{Id}_{\mathfrak{n}} - \operatorname{Ric})  \operatorname{ad}_{e_j}  (\lambda_j \operatorname{Id}_{\mathfrak{n}} - \operatorname{Ric})$ and using the cyclic invariance of the trace ($\operatorname{tr}(A B) = \operatorname{tr}(B A)$), we find
\begin{equation*}
    \operatorname{tr}(T_j^2) = \lambda_j^2 \operatorname{tr}(\operatorname{ad}_{e_j}^2) - 2\lambda_j \operatorname{tr}(\operatorname{ad}_{e_j}^2  \operatorname{Ric}) + \operatorname{tr}(\operatorname{ad}_{e_j}  \operatorname{Ric} \operatorname{ad}_{e_j}  \operatorname{Ric}).
\end{equation*}
Since $\mathfrak{n}$ is nilpotent, $\operatorname{tr}(\operatorname{ad}_{e_j}^2) = 0$ and therefore, by \eqref{eq:trace_conditions}, $\operatorname{tr}(T_j^2) = 0$. Summing over $j = 1, \dots, n$, we conclude from \eqref{eq:squared_cyclic_sum} that $\sum_{i,j,k=1}^n (\lambda_i - \lambda_j)^2 C_{i,j,k}^2 = 0$. Hence, $(\lambda_i - \lambda_j) C_{i,j,k} = 0$ for all $i, j, k$.

This shows that if the basis elements $e_i$ and $e_j$ correspond to distinct eigenvalues ($\lambda_i \neq \lambda_j$), their Lie bracket must vanish ($[e_i, e_j] = 0$). By bilinearity, it follows that the Lie bracket between any two eigenspaces of the Ricci operator associated with distinct eigenvalues is trivial. Finally, by Proposition \ref{prop:qc_diagonalizable_iff}, this orthogonal commutativity of eigenspaces is precisely equivalent to $\operatorname{Ric} \in \operatorname{QC}(\mathfrak{n})$.
\end{proof}

\begin{remark}
The first trace condition in \eqref{eq:trace_conditions} is interesting in the following sense: it is a well-known fact that the Ricci operator of a nilpotent Riemannian Lie algebra $(\mathfrak{n}, \langle \cdot, \cdot \rangle)$ is orthogonal to the derivation algebra $\operatorname{Der}(\mathfrak{n})$ with respect to the Hilbert-Schmidt inner product. Since $\operatorname{Ric}$ is self-adjoint and $\operatorname{ad}_{[X,Y]} = [\operatorname{ad}_X, \operatorname{ad}_Y]$ is an inner derivation, this orthogonality means exactly that $\operatorname{tr}(\operatorname{ad}_{[X,Y]} \operatorname{Ric}) = 0$ for all $X, Y \in \mathfrak{n}$. This immediately implies that
\begin{equation*}
    \operatorname{tr}(\operatorname{ad}_X \operatorname{ad}_Y \operatorname{Ric}) = \operatorname{tr}(\operatorname{ad}_Y \operatorname{ad}_X \operatorname{Ric}).
\end{equation*}
Therefore, the function $\theta(X, Y) := \operatorname{tr}(\operatorname{ad}_X \operatorname{ad}_Y \operatorname{Ric})$ defines a symmetric bilinear form on $\mathfrak{n}$.
\end{remark}

\section*{Acknowledgements}
The author is profoundly indebted to Nadina Rojas, not only for the many stimulating discussions that helped shape these ideas, but above all for her constant and invaluable support.

Finally, the author acknowledges the use of artificial intelligence (Gemini) strictly to polish the mathematical English, improve phrasing, and ensure stylistic consistency across the manuscript. All mathematical arguments, logical deductions, and original proofs presented in this manuscript are the work of the author, with standard concepts and adapted results properly attributed.

\section*{Declarations}
\noindent \textbf{Data availability.} No data was used for the research described in this article.

\vspace{0.2cm}
\noindent \textbf{Conflict of interest.} The author declares no competing interests.

\end{document}